\documentclass[12pt,oneside]{amsart}

\usepackage{amsfonts}
\usepackage[top=1in, bottom=1in,left=1in,right=1in]{geometry}
\usepackage{epsfig,afterpage}
\usepackage[dvips]{psfrag}
\usepackage{psfrag}
\usepackage[all]{xy}
\usepackage{color}
\usepackage{graphicx}
\usepackage{comment}

\newtheorem{theorem}{Theorem}[section]
\newtheorem{definition}[theorem]{Definition}
\newtheorem{prop}[theorem]{Proposition}
\newtheorem{lemma}[theorem]{Lemma}
\newtheorem{cor}[theorem]{Corollary}
\newtheorem{remark}[theorem]{Remark}

\numberwithin{equation}{section}

\newtheorem{thmx}{Theorem}

\usepackage{tikz}
\usetikzlibrary{arrows}
\usetikzlibrary{patterns,decorations.pathreplacing}

\usepackage{epsfig,afterpage}
\usepackage[dvips]{psfrag}
\usepackage{psfrag}

\title{Hyperbolicity of renormalization for dissipative gap mappings}
\date{\today}


\author{Trevor Clark} 
\address{Trevor Clark, Imperial College London, London, UK}
\email[]{t.clark@imperial.ac.uk}

\author{M\'arcio Gouveia} 
\address{M\'arcio Gouveia,
IBILCE-UNESP, CEP 15054-000, S. J. Rio Preto,
S\~{a}o Paulo, Brazil}
\email[]{mra.gouveia@unesp.br}

\thanks{This work has been partially supported by 
ERC AdG grant no 339523 RGDD; 
EU Marie-Curie IRSES Brazilian-European partnership in Dynamical Systems (FP7-PEOPLE-2012-IRSES 318999 BREUDS); FAPESP grants 2017/25955-4 and 2013/24541-0 and CAPES}

\begin{document}

\maketitle

\begin{abstract} 
A gap mapping is a discontinuous interval mapping with two strictly increasing branches that have a gap between their ranges. They are one-dimensional dynamical systems, which arise in the study of certain higher dimensional flows, for example the Lorenz flow and the Cherry flow.
In this paper we prove hyperbolicity of renormalization acting on $C^3$ dissipative gap mappings, and show that the topological conjugacy classes of infinitely renormalizable gap mappings are $C^1$ manifolds.
\end{abstract}

\section{Introduction}
Higher dimensional, physically relevant, dynamical systems often possess features that can be studied using techniques from one-dimensional dynamical systems. Indeed, often a one-dimensional discrete dynamical system captures essential features of a higher dimensional flow. For example, for the Lorenz flow, whose dynamics were first studied in \cite{L}, one may study the return mapping to a plane transverse to its stable manifold, the stable manifold intersects the plane in a curve, and the return mapping to this curve is a (discontinuous) one-dimensional dynamical system known as a Lorenz mapping, see \cite{Tucker}.
This approach has been very fruitful in the study of the Lorenz flow.
It would be difficult to cite all the papers studying this famous dynamical system, but for example see \cite{ABS, Wi, GuWi, ACT,GPTT, Rov}.
The success of the use of the one-dimensional Lorenz mapping in studying the flow has led to an extensive study of these interval mappings, see
\cite{StP, KP, MdM, LM, W,GW,MarcoWinclker, MW2, Bran} among many others.
Great progress in understanding the Cherry flow on a two-torus has followed from a similar approach
\cite{Cherry, MvSdMM, Mendes,AZM, dM, NZ, SV,P4,P3,P2,P1}.

In this paper we study a class of Lorenz mappings, which have ``gaps" in their ranges. These mappings arise as return mappings for the Lorenz flow and
for certain Cherry flows. They are also among the first examples of mappings with a wandering interval - the gap. This phenomenon is ruled out for $\mathcal C^{1+\mathrm{Zygmund}}$ mappings with a non-flat critical point by \cite{vSV}. In fact, in \cite{BM} it is proved that Lorenz mappings satisfying a certain bounded non-linearity condition have a wandering interval if and only if they have a renormalization which is a gap mapping. See the introduction of \cite{AlvesColli2} for detailed history of gap mappings.

The main result of this paper concerns the structure of the topological conjugacy classes of
$\mathcal C^4$ dissipative gap mappings. Roughly, these are discontinuous mappings with two orientation preserving branches, whose derivatives are bounded between zero and one. They are defined in Definition~\ref{def:dis gap map}. 

\begin{thmx}
\label{thm:A}
The topological conjugacy class of an infinitely renormalizable $\mathcal C^4$ dissipative gap mapping is a $\mathcal C^1$-manifold of codimension-one in the space of dissipative gap maps.
\end{thmx}

To obtain this result, we prove the hyperbolicity of renormalization for dissipative gap mappings. 
In the usual approach to renormalization, one considers renormalization as a restriction of a high iterate of a mapping. While this is conceptually straightforward, it is technically challenging as the composition operator acting on the space of, say, $\mathcal C^4$ functions is not differentiable. 
Nevertheless, we are able to show that the tangent space admits a hyperbolic splitting.
To do this, we work in the {\em decomposition space} introduced by Martens in \cite{Marco},
see Section~\ref{sec:decomp} of this paper for the necessary background.

\begin{thmx}
The renormalization operator $\mathcal R$ acting on the space of dissipative gap mappings has a hyperbolic splitting. More precisely, if $f$ is an infinitely renormalizable $\mathcal C^3$ dissipative gap mapping then for any
$\delta\in(0,1),$ and for all $n$
sufficiently big,
the derivative of the renormalization operator
acting on the decomposition space $\underline{\mathcal{D}}$
satisfies the following:
\begin{itemize}
\item $T_{\underline{\mathcal{R}}_{\underline{\mathcal{R}}^n\underline f}} \underline{\mathcal{D}}=E^u\oplus E^s,$ and the subspace $E^u$ is one dimensional. 

\item For any vector $v\in E^u$, we have that 
$\|D\underline{\mathcal{R}}_{\underline{\mathcal{R}}^n\underline f}v\|\geq \lambda_1\|v\|$, where $|\lambda_1|>1/\delta$.
\item For any
 $v\in E^s$, we have that 
$\|D\underline{\mathcal{R}}_{\underline{\mathcal{R}}^n\underline f}v\|\leq \lambda\|v\|$, where $|\lambda|<\delta$.
\end{itemize}
\end{thmx}

Gap mappings can be regarded as discontinuous circle mappings, and indeed they have a well-defined rotation number \cite{Br}, and they are infinitely renormalizable precisely when the rotation number is irrational. Consequently, from a combinatorial point of view they are similar to critical circle mappings. However, unlike critical circle mappings, the geometry of gap mappings is unbounded.
For example, for critical circle mappings
the quotient of the lengths of successive renormalization intervals is bounded away from zero and infinity \cite{dFdM}, but for gap mappings it diverges very fast \cite{AlvesColli2}. As a result, the renormalization operator for gap mappings does not seem to possess a natural extension to the limits of renormalization ({\em c.f.} \cite{MP}).

Renormalization theory was introduced into dynamical systems from statistical physics by 
Feigenbaum \cite{Feigenbaum} and Coullet-Tresser \cite{Coullet-Tresser,TC}
in the 1970's to explain the universality phenomena they observed in the quadratic family. 
They conjectured that the period-doubling renormalization operator acting on an appropriate space of analytic unimodal mappings is hyperbolic. The first proof of this conjecture was obtained using computer assistance in \cite{Lanford}.
The conjecture can be extended to all combinatorial types, and to multimodal mappings.
A conceptual proof was given for
analytic unimodal mappings of any combinatorial type in the work of Sullivan \cite{Sullivan} (see also \cite{dMvS}), 
McMullen \cite{McM1, McM2}, 
Lyubich \cite{Lyubich1, Lyubich2} 
and Avila-Lyubich \cite{Avila-Lyubich}.
This was extended to certain
smooth mappings in \cite{dFdMP},
and to analytic mappings with several critical points and bounded combinatorics by Smania \cite{Smania1,Smania2}.
Renormalization is intimately related with rigidity theory, and in many contexts, 
{\em e.g.} interval mappings and critical circle mappings,
exponential convergence of renormalization implies that two topologically conjugate infinitely renormalizable mappings are smoothly conjugate on their (measure-theoretic) attractors. However, for gap mappings
it is not the case that exponential convergence of renormalization implies rigidity; indeed, in general one can not expect topologically conjugate gap mappings to be $\mathcal C^1$ conjugate \cite{AlvesColli2}.

The aforementioned results on renormalization of interval mappings all depend on complex analytic tools, and consequently, many of the
tools developed in these works can only be applied to mappings with a critical point of integer order. 
The goal of studying mappings with arbitrary critical order was one of Martens' motivations for introducing the decomposition space, mentioned above.
This purely real approach has led to results on the renormalization in various contexts.
In \cite{Marco} this approach was used to establish the existence of periodic points of renormalization of any combinatorial type for unimodal mappings 
$x\mapsto x^\alpha+c$ where $\alpha>1$ is not necessarily an integer. For Lorenz mappings of certain monotone combinatorial types, \cite{MarcoWinclker} 
proved that there exists a global two-dimensional strong unstable manifold at every point in the limit set of renormalization using this approach.
In \cite{MP} they studied renormalization acting on the decomposition space for infinitely renormalizable critical circle mappings with a flat interval. They proved that for certain mappings with stationary, Fibonacci, combinatorics that the renormalization operator is hyperbolic, and that the class of mappings with Fibonacci combinatorics is a $\mathcal C^1$ manifold. 

Analytic gap mappings were studied in \cite{AlvesColli, AlvesColli2} using different methods to those that we use here. 
In the former paper, they proved hyperbolicity of renormalization in the special case of affine dissipative gap mappings, and in the latter paper,
they proved that the topological conjugacy classes of analytic infinitely renormalizable dissipative gap mappings are analytic manifolds.
We appropriately generalize these two results to the $\mathcal C^4$ case. Since the renormalization operator does not extend to the limits of renormalization, it seems to be difficult to build on the hyperbolicity result for affine mappings to extend 
it to smooth mappings (similarly to what was done in \cite{dFdMP}), and so we follow a different approach.
In \cite{AlvesColli2}, it is also proved that
two topologically conjugate dissipative gap mappings are H\"older conjugate. 
We improve this rigidity result, and give a simple proof that topologically conjugate dissipative gap mappings are quasisymmetrically conjugate, see Proposition~\ref{prop:qs}. 

This paper is organized as follows: In Section 2 we will provide the necessary background material on gap mappings, and in Section 3 we will describe the decomposition space of infinitely renormalizable gap mappings.
The estimate of the derivative of renormalization operator is done in Section~\ref{sec:derivative}, and it is the key technical result of our work.
In our setting we are able to obtain fairly complete results without any restrictions on the combinatorics of the mappings.
In Section 5 we use the estimates of Section 4 and ideas from \cite{MP} to show that the renormalization operator is hyperbolic and that
the conjugacy classes of dissipative gap mappings are $\mathcal C^1$ manifolds.

\section{Preliminaries}

\subsection{The dynamics of gap maps}


In this section we collect the necessary background material on gap mappings, see
 \cite{AlvesColli2} for further results.

\medskip

A {\em Lorenz map} is a function $f:[a_L, a_R] \setminus \{ 0 \} \rightarrow [a_L, a_R]$ satisfying:
\begin{enumerate}
    \item[(i)] $a_L< 0 < a_R$;
    \item[(ii)] $f$ is continuous and strictly increasing in the intervals $[a_L,0)$ and $(0,a_R]$;
    \item[(iii)] the left and right limits at $0$ are $f(0^-)=a_R$ and $f(0^+)=a_L$.
\end{enumerate}
A {\em gap map} is a Lorenz map $f$ that is not surjective, i.e. a map satisfaying conditions (i), (ii), (iii) with $f(a_L) > f(a_R)$. In this case the {\em gap} is the interval $G_f=(f(a_R), f(a_L))$. When it will not cause confusion, we omit the subscript and denote the gap by $G$, see Figure~\ref{fig: slopes and branches of f}.

\begin{definition}
\label{def:dis gap map}
A dissipative gap map is a gap map $f$ that is differentiable in $[a_L, a_R] \setminus \{ 0 \}$ and satisfies: $0 < f'(x) \leq \nu$ for every $x \in [a_L, a_R] \setminus \{ 0 \}$, and for some real number $\nu = \nu_f \in (0,1)$. 
\end{definition}

The {\em space of dissipative gap maps} is defined in the following way. Consider

\begin{equation}
\begin{array}{ccl}
   \mathcal{D}_L^k  & = & \{ u_L:[-1,0) \rightarrow \mathbb{R}; \; u_L\in\mathrm{Diff}_+^k[-1,0],   \\
     & & u_L(0^-)=0, \; \mbox{and} \; \exists \nu \in (0,1) \; \mbox{s.t.} \; 0 < u_L'(x) \leq \nu, \; \forall x \in [-1,0) \},
\end{array}    
\end{equation}

\begin{equation}
\begin{array}{ccl}
   \mathcal{D}_R^k  & = & \{ u_R:(0, +1] \rightarrow \mathbb{R}; \; u_R \in\mathrm{Diff}_+^k[0,1],  \\
     & & u_R(0^+)=0, \; \mbox{and} \; \exists \nu \in (0,1) \; \mbox{s.t.} \; 0 < u_R'(x) \leq \nu, \; \forall x \in [-1,0) \},
\end{array}    
\end{equation}
and $\mathcal{D}^k = \mathcal{D}_L^k \times \mathcal{D}_R^k \times (0,1)$, where $\mathrm{Diff}_+^k[a,b]$ denotes the space of orientation preserving $\mathcal C^k$ diffeomorphisms on $(a,b)$, which are continuous on
$[a,b]$. We will always assume that $k\geq 3,$ and unless otherwise stated, the reader can assume that $k=3.$

For each element $(u_L, u_R, b) \in \mathcal{D}^k$ we associate a function $f:[-1,1] \setminus \{ 0 \} \rightarrow [-1,1]$ defined by

\begin{equation}
\label{definition of a typical gap map f}
f(x) = \left\{ \begin{array}{lcl}
u_L(x)+b    & ,& x \in [-1,0) \\
u_R(x)+b-1  & ,& x \in (0,+1]
\end{array}    \right.,
\end{equation}
and take $\nu = \nu_f \in (0,1)$ that bounds the derivative on each branch from above. It is not difficult to check that the interval $[b-1,b]$ is invariant under $f$ and $f$ restricted to $[b-1, b] \setminus \{ 0 \}$ is a dissipative gap map. For the sake of simplicity, we write $f=(u_L, u_R, b)$, and we use the following notations for the left and right branches of $f$:

\begin{equation}
\label{branches of f}
\begin{array}{clcl}
    f_L(x) & = u_L(x)+b &,& x <0\\
    f_R(x) & = u_R(x)+b-1 &,& x>0
\end{array}    
\end{equation}

\smallskip 

We endow $\mathcal{D}^k = \mathcal{D}_L^k \times \mathcal{D}_R^k \times (0,1)$ with the product topology. 

\begin{definition}
\label{sign of f}
Let $f:[a_L, a_R] \setminus \{0 \} \rightarrow [a_L, a_R]$ be a dissipative gap map. We define the {sign} of $f$ by
\begin{equation}
\sigma_f:= \left\{
\begin{array}{cc}
-    & \mbox{ if } a_R \leq |a_L|\\
+    & \mbox{ if } a_R > |a_L|.
\end{array}    \right.
\end{equation}
\end{definition}
It is an easy consequence of this definition that for a dissipative gap map $f$ we have $\sigma_f=-$ if $G \subset [a_L,0)$ and $\sigma_f=+$ when $G \subset (0,a_R]$.
\subsection{Renormalization of dissipative gap mappings}
\begin{definition}
\label{renormalizable gap map}
A dissipative gap map $f:[a_L, a_R] \setminus \{0 \} \rightarrow [a_L, a_R]$ is renormalizable if there exists a positive integer $k$ such that 
\begin{enumerate}
    \item[(a)] $0 \notin \cup_{i=0}^k \overline{f^i(G)}$;
    \item[(b)] either $\overline{G}, \overline{f(G)}, \ldots, \overline{f^{k-1}(G)} \subset (b-1,0)$ and $\overline{f^k(G)} \subset (0,b)$, \linebreak or $\overline{G}, \overline{f(G)}, \ldots, \overline{f^{k-1}(G)} \subset (0,b)$ and $\overline{f^k(G)} \subset (b-1,0)$.
\end{enumerate}
 
\end{definition}

\begin{remark}
The positive number $k$ in Definition~\ref{renormalizable gap map} is chosen to be minimal so that (a) and (b) hold.
\end{remark}

\begin{definition}
\label{1st return map}
Let $f:[a_L, a_R] \setminus \{0 \} \rightarrow [a_L, a_R]$ be a renormalizable dissipative gap map, and consider $I'= [a_L', a_R' ]=I_f'$ the interval containing $0$ whose boundary points are the boundary points of $f^{k-1}(G)$ and $f^k(G)$ which are nearest to $0$, that is

\begin{equation}
\label{domain of R}
    \begin{array}{ccl}
    I'= &  [f^k(a_L), f^{k+1}(a_R)] & \mbox{ for } \sigma_f = -\\
    I'= & [f^{k+1}(a_L), f^k(a_R)] & \mbox{ for } \sigma_f=+.
    \end{array}
\end{equation}
The first return map $R=R_f$ to $I'$ is given by

\begin{equation}
R(x)= \left\{    \begin{array}{cl}
f^{k+2}(x)     &  \mbox{ if } x \in [f^k(a_L),0)\\
f^{k+1}(x)     &  \mbox{ if } x \in (0,f^{k+1}(a_R)]
\end{array}\right.,     
\end{equation}
in the case $\sigma_f=-$, and 

\begin{equation}
R(x)= \left\{    \begin{array}{cl}
f^{k+1}(x)     &  \mbox{ if } x \in [f^{k+1}(a_L),0)\\
f^{k+2}(x)     &  \mbox{ if } x \in (0,f^{k}(a_R)]
\end{array}\right.,     
\end{equation}
in the case $\sigma_f=+$. The renormalization of $f$, $\mathcal{R}f$, is the first return map $R$ rescaled and normalized to the interval $[-1,1]$ and given by
\begin{equation}
\mathcal{R}f(x)= \frac{1}{|I'|}R(|I'|x)
\end{equation}
for every $x \in [-1,1] \setminus \{0 \}$.
\end{definition}
In terms of the branches $f_L$ and $f_R$ defined in (\ref{branches of f}) the first return map $R$ is given by

\begin{equation}
R(x)= \left\{    \begin{array}{ll}
f_L^k \circ f_R \circ f_L (x)     &  \mbox{ if } x \in [f^k(a_L),0)\\
f_L^k \circ f_R (x)     &  \mbox{ if } x \in (0,f^{k+1}(a_R)]
\end{array}\right.,     
\end{equation}
in the case $\sigma_f=-$, and 

\begin{equation}
R(x)= \left\{    \begin{array}{ll}
f_R^k \circ f_L (x)     &  \mbox{ if } x \in [f^{k+1}(a_L),0)\\
f_R^k \circ f_L \circ f_R (x)     &  \mbox{ if } x \in (0,f^{k}(a_R)]
\end{array}\right.,     
\end{equation}
in the case $ \sigma_f=+ $.

From Definition~\ref{1st return map} we have a natural operator which sends a renormalizable dissipative gap map $f$ to its renormalization $\mathcal{R}f$, which is also a dissipative gap map:

\begin{definition}
\label{renormalization operator}
The renormalization operator is defined by
\begin{equation}
\begin{array}{cccll}
     \mathcal{R} & : & \mathcal{D}_{\mathcal{R}}^k & \rightarrow & \mathcal{D}^k  \\
     & & f & \mapsto & \mathcal{R}f 
\end{array}    
\end{equation}
where $\mathcal{R}f(x)= \frac{1}{|I'|}R(|I'|x)$, and $\mathcal{D}_{\mathcal{R}}^k \subset \mathcal{D}^k$ is the subset of all renormalizable dissipative gap maps in $\mathcal{D}^k$.
\end{definition}

From now on, we assume that the interval $[a_L, a_R]$ has size $1$. 

Although a dissipative gap map is not defined at $0$ we define the lateral orbits of $0$ taking $0_j^+ = f^j(0^+)= \lim_{x \rightarrow 0^+}f^j(x)$ and $0_j^- = f^j(0^-)= \lim_{x \rightarrow 0^-}f^j(x)$. We first observe that $0_j^+=f^{j-1}(b-1)$ and $0_j^-=f^{j-1}(b)$. The left and right future orbits of $0$ are the sequences $(0_j^+)_{j \geq 1}$ and $(0_j^-)_{j \geq 1}$ which are always defined unless there exists $j \geq 1$ such that either $0_j^+ = 0$ or $0_j^-=0$. Using this notation for the interval $I'$ defined in (\ref{domain of R}), we obtain 

\begin{equation}
\label{domain of R again}
    \begin{array}{ccl}
    I'= &  [0_{k+1}^+,0_{k+2}^-] = [f_L^k (b-1),f_L^k \circ f_R(b)]& \mbox{ for } \sigma_f = -\\
    && \\
    I'= & [0_{k+2}^+,0_{k+1}^-] = [f_R^k \circ f_L(b-1),f_R^k (b)]& \mbox{ for } \sigma_f=+.
    \end{array}
\end{equation}
See Figure~\ref{fig central interval} for an illustration of one example of case with $\sigma_f=-$. 

One can show inductively that for each gap mapping $f$ there are $n=n(f) \in \{ 0, 1, 2, \ldots \} \cup \{ \infty \} $ and a sequence of nested intervals $(I_i)_{0 \leq i < n+1}$, each one containing $0$, such that:

\begin{enumerate}
    \item[1.] the first return map $R_i$ to $I_i$ is a dissipative gap map, for every $0 \leq i <n+1$;
    
    \item[2.] $I_{i+1} = I_{R_i}'$, for every $0 \leq i <n$;
\end{enumerate}

If $n< \infty $ we say that $f$ is {\it finitely renormalizable} and $n$ {\it times renormalizable}, and if $n = \infty $ we say that $f$ is {\it infinitely renormalizable}. Moreover, we call $G_i = G_{R_i}$, $\sigma_i = \sigma_{R_i}$ and $k_i = k_{R_i}$, for every $0 \leq i < n+1$. In particular, this defines the {\it combinatorics} $\Gamma = \Gamma (f)$ for $f$, given by the (finite or infinite) sequence

\begin{equation}
    \Gamma = ((\sigma_i, k_i))_{1 \leq i < n+1}.
\end{equation}
\begin{prop}\cite{AlvesColli2}
Two infinitely renormalizable dissipative gap mappings that have the same combinatorics are topologically conjugate.
\end{prop}

For more details about this inductive definition and related properties see \cite{AlvesColli2}.

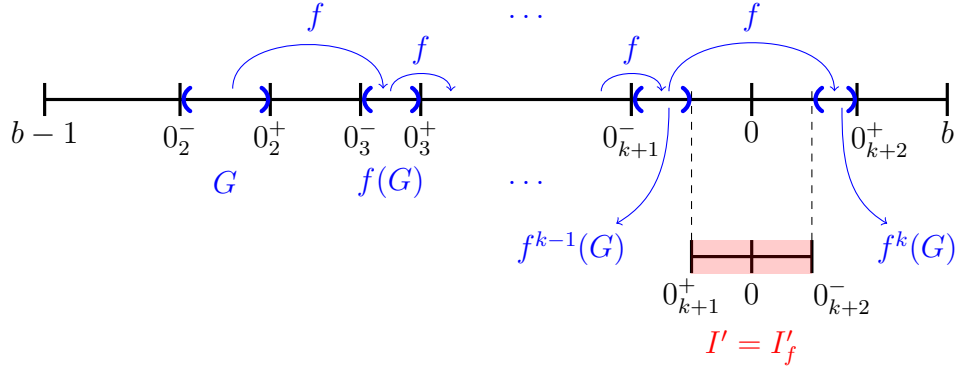
\begin{figure}
\label{fig central interval}
	\begin{center}
\begin{tikzpicture}[mydrawstyle/.style={draw=black, very thick}, node distance=0.5cm, x=2mm, y=1.1mm, z=1mm]

\node (1) {};
\node (2) [right of=1] {};
\node (3) [right of=2] {};
\node (4) [right of=3] {};
\node (5) [right of=4] {};
\node (6) [right of=5] {};
\node (7) [right of=6] {};
\node (8) [right of=7] {};
\node (9) [right of=8] {};
\node (10) [right of=9] {};
\node (11) [right of=10] {};
\node (12) [right of=11] {};
\node (13) [right of=12] {};
\node (14) [right of=13] {};
\node (15) [right of=14] {};
\node (16) [right of=15] {};
\node (17) [right of=16] {};
\node (18) [right of=17] {};

\draw [->,blue] (6.north) to [out=90,in=90] (10.north); 

\draw [->,blue] (23,1) to [out=90,in=90] (27,1);
\draw [blue] ( 25,6) node {$f$};

\draw [->,blue] (37,1) to [out=90,in=90] (41,1);
\draw [blue] ( 39,6) node {$f$};

\draw [->,blue] (41.5,1) to [out=90,in=90] (52.5,1);
\draw [blue] (47,10) node {$f$};

\draw [->,blue] (41.5,-1) to [out=270,in=30] (38,-15);
\draw [->,blue] (53,-1) to [out=270,in=150] (56,-15);

\draw[mydrawstyle, -](0,0)--(60,0) node at (-6,0)[left]{};
\draw[mydrawstyle](0,-2)--(0,2) node[below=10]{$b-1$};
\draw[mydrawstyle](9,-2)--(9,2) node[below=10]{$0_2^-$};
\draw[mydrawstyle](15,-2)--(15,2) node[below=10]{$0_2^+$};
\draw [blue] (12,-10) node {$G$};
\draw [blue] (18,10) node {$f$};

\draw[mydrawstyle](21,-2)--(21,2) node[below=10]{$0_3^-$};
\draw[mydrawstyle](25,-2)--(25,2) node[below=10]{$0_3^+$};
\draw [blue] (23,-10) node {$f(G)$};

\draw [blue] (32,-10) node {$\ldots$};
\draw [blue] (32,10)  node {$\ldots$};

\draw[mydrawstyle](39,-2)--(39,2) node[below=10]{$0_{k+1}^-$};
\draw [blue] (35,-18) node {$f^{k-1}(G)$};

\draw [black] (55.5,-5) node {$0_{k+2}^+$};

\draw[mydrawstyle](54,-2)--(54,2) node[below=10]{};
\draw [blue] (58,-18) node {$f^{k}(G)$};

\draw[mydrawstyle](47,-2)--(47,2) node[below=10]{$0$};
\draw[mydrawstyle](60,-2)--(60,2) node[below=10]{$b$};

 \draw [dashed] (43,-19) -- (43,1);
 \draw [dashed] (51,-19) -- (51,1);
 \draw[mydrawstyle, -](43,-19)--(51,-19) node at (-6,0)[left]{};
 
 \draw[mydrawstyle](47,-21)--(47,-17) node[below=12]{$0$};
 \draw[mydrawstyle](43,-21)--(43,-17) node[below=10]{$0_{k+1}^+$};
 \draw[mydrawstyle](51,-21)--(51,-17) node[below=10]{};
 \draw [black] (53,-24) node {$0_{k+2}^-$};

\draw [red] (47,-30) node {$I'=I'_f$};

\fill[fill=red, opacity=0.2](43,-21)--(43,-17)--(51,-17)--(51,-21)--(43,-21);

\draw[(-, ultra thick, blue] (9.1,0) -- (9.3,0);
\draw[-), ultra thick, blue] (14.8,0) -- (15,0);

\draw[(-, ultra thick, blue] (21.1,0) -- (21.3,0);
\draw[-), ultra thick, blue] (24.8,0) -- (25,0);

\draw[(-, ultra thick, blue] (39.1,0) -- (39.3,0);
\draw[-), ultra thick, blue] (42.8,0) -- (43,0);

\draw[(-, ultra thick, blue] (51.1,0) -- (51.3,0);
\draw[-), ultra thick, blue] (53.8,0) -- (54,0);

\end{tikzpicture}
\caption{$I'$: the domain of the first return map $R$ in case $\sigma=-$.}
\end{center}
\end{figure}

\subsection{Quasisymmetric rigidity}
We know that two dissipative gap mappings with the same irrational rotation number are
H\"older conjugate \cite[Theorem A]{AlvesColli2}; however, more is true. Let $\kappa\geq 1$, and let $I$ denote an interval in $\mathbb R$. Recall, that a mapping $h:I\to I$ is $\kappa$-{\it quasisymmetric} if for any $x\in I$ and $a>0$ so that
$x-a$ and $x+a$ are in $I$, we have
$$\frac{1}{\kappa}\leq\frac{|h(x+a)-h(x)|}{|h(x)-h(x-a)|}\leq \kappa.$$

\begin{prop}
\label{prop:qs}
Suppose that $f,g$ are two dissipative gap maps with the same irrational rotation
number, then $f$ and $g$ are quasisymmetrically conjugate.
\end{prop}
\begin{proof}
Let $\phi,\psi$ denote $f^{-1},g^{-1}$, respectively.
Then $\phi$ and $\psi$ can be extended to
expanding, degree three, covering maps of the circle, which we will
continue to denote by $\phi$ and $\psi$.
These extended mappings are topologically conjugate, and so they 
are quasisymmetrically conjugate. To see this, one may argue exactly
as described in II.2, Exercise 2.3 of \cite{dMvS}.
Thus there exists a quasisymmetric mapping $h$ of the circle so that 
$h\circ \phi(z)=\psi\circ h(z)$. Thus we have that 
$h^{-1}\circ g=f\circ h^{-1},$ and it is well known that the inverse of a quasisymmetric mapping is quasisymmetric.
\end{proof}

\subsection{Convergence of renormalization to affine maps}
\begin{prop}
Suppose that $f$ is an infinitely renormalizable
dissipative gap mapping. Then 
for any $\varepsilon>0$ there exists $n_0\in\mathbb N$
so that for all $n\geq n_0,$ there exists an affine gap mapping
$g_n$, so that
$\|R^n f-g_n\|_{\mathcal C^3}\leq\varepsilon.$
\end{prop}
\begin{proof}
Let us recall the formulas for the nonlinearity, $N$ and Schwarzian derivative, $S$, of iterates of $f$:
$$Nf^k(x)=\sum_{i=0}^{k-1}Nf(f^i(x))|Df^i(x)|,$$ and
$$ Sf^k(x) = \sum_{i=0}^{k-1}Sf^i(x)|Df^i(x)|^2.$$
Since the derivative of $f$ is bounded away from one, these quantities are bounded in terms of $Nf$ and $Sf$, respectively.
But now, since $|Nf|$ is bounded, say by $C_1>0$ we have that there exists $C_2>0$ so that
$$|Nf^k|=\left|\frac{D^2f^k}{Df^k}\right|<C_2.$$
Since $Df^k\rightarrow 0,$ as $k$ tends to $\infty$, so does $D^2f^k$.

Now, 
$$Sf^k=\frac{D^3f^k}{Df^k}-\frac{3}{2}(Nf^k)^2,$$
and arguing in the same way, we have that $D^3f^k\to 0$ as $k\to\infty.$ Thus by taking $k$ large enough, $f^k$ is arbitrarily close to its affine part in the $\mathcal C^3$-topology.
\end{proof}

\section{Renormalization of decomposed mappings}
\label{sec:decomp}

In this section we recall some background material on the nonlinearity operator and decomposition spaces; for further details see \cite{Marco, MarcoWinclker}. We then define the decomposition space of dissipative gap mappings, and describe the action of renormalization on this space.



\subsection{The nonlinearity operator}

\begin{definition}
\label{def nonlinearity operator}
The nonlinearity operator 
$N:\mathrm{Diff}_+^k([0,1]) \rightarrow \mathcal{C}^{k-2}([0,1])$ is defined by
\begin{equation}
\label{expression nonlinearity operator}  
N \varphi := D \log D \varphi = \frac{D^2 \varphi}{D \varphi},
\end{equation}
and $N \varphi$ is called the nonlinearity of $\varphi$.
\end{definition}

\begin{remark}
For convenience we use the abbreviated notation
\[
N \varphi = \eta_{\varphi}.
\]
\end{remark}

\begin{lemma}
\label{lemma N is a bijection}
The nonlinearity operator is a bijection.
\end{lemma}

\begin{proof}
The operator $N$ has an explicit inverse given by

\[
N^{-1}f(x) = \frac{\int_0^x e^{\int_0^s f(t)dt}ds}{\int_0^1 e^{\int_0^s f(t)dt}ds},
\]
where $f \in \mathcal{C}^0([0,1])$.
\end{proof}

By Lemma~\ref{lemma N is a bijection}, we can
identify $\mbox{Diff}_+^{3}([0,1])$ with
$\mathcal{C}^{1}([0,1])$ using the nonlinearity operator. It will be convenient to work with the  norm induced on 
$\mbox{Diff}_+^3([0,1])$ by this identification. 
For $\varphi\in\mathrm{Diff}_+^3([0,1])$, we define
\[
||\varphi|| = ||N\varphi||_{\mathcal{C}^1} = ||\eta_{\varphi}||_{\mathcal{C}^1}.
\]

We say that a set $T$ is a {\em time set} if it is at most countable and totally ordered.
Given a time set $T$,
let $X$ denote the {\em space of decomposed diffeomorphisms labelled by $T$}:

\[
X = \{ \underline{\varphi} = \big( \varphi_n \big)_{n \in T}; \; \varphi_n \in \mbox{Diff}_{+}^{3}([0, 1]), \mbox{ and }  \sum ||\varphi_n|| < \infty \}.
\]
The norm of an element $\underline{\varphi} \in X $ is defined by

\[
||\underline{\varphi}|| = \displaystyle \sum_{n\in T} ||\varphi_n||.
\]

Given two time sets $T_1$ and $T_2$, we define
$$T_2\oplus T_1=\{(x,i):x\in T_i,i=1,2\},$$
where $(x,i)<(y,i)$ if and only if $x<y,$
and $(x,2)>(y,1)$ for all $x\in T_2,$ $y\in T_1.$

To simplify the following discussion, assume that
$T=\{1,2,3,\dots,n\}$ or $T=\mathbb N$.
We define the partial composition by

\begin{equation}
\begin{array}{ccccl}
  O_n & : & X & \rightarrow & \mbox{Diff}_{+}^{2}([0, 1]) \\
     & & \underline{\varphi} & \mapsto & O_n \underline{\varphi} := \varphi_n \circ \varphi_{n-1} \circ \ldots \circ \varphi_1 \\
\end{array}    
\end{equation}
and the complete composition is given by the limit
\begin{equation}
\label{operator: composition of a decomposition}
\displaystyle O \underline{\varphi} = \lim_{n \rightarrow \infty} O_n \underline{\varphi}    
\end{equation}
which allow us to define the operator

\begin{equation}
\label{complete composition operator}
\begin{array}{ccccl}
  O & : & X & \rightarrow & \mbox{Diff}_{+}^{2}([0, 1]) \\
     & & \underline{\varphi} & \mapsto & O \underline{\varphi}:= \displaystyle \lim_{n \rightarrow \infty } O_n \underline{\varphi}. \\
\end{array}    
\end{equation}
The existence of the limit (\ref{operator: composition of a decomposition}) is assured by the Sandwich Lemma from \cite{Marco}. 

\subsection{The decomposition space for dissipative gap mappings}

We define the decomposition space of dissipative gap maps, $\underline{\mathcal{D}}$, by
\[
\underline{\mathcal{D}} = (0,1)^3 \times X \times X.
\]
The composition operator defined at (\ref{complete composition operator}) gives a way to project the space $\underline{\mathcal{D}}$ to the space $(0,1)^3 \times \mbox{Diff}_{+}^{2}([0, 1]) \times \mbox{Diff}_{+}^{2}([0, 1])$. More precisely

\begin{equation}
\begin{array}{ccccl}
   \Xi  & : & \underline{\mathcal{D}} & \rightarrow &  (0,1)^3 \times \mbox{Diff}_{+}^{2}([0, 1]) \times \mbox{Diff}_{+}^{2}([0, 1]) \\
  & & (\alpha, \beta, b, \underline{\varphi}_L, \underline{\varphi}_R) & \mapsto & \Xi (\alpha, \beta, b, \underline{\varphi}_L, \underline{\varphi}_R) : = (\alpha, \beta, b, O \underline{\varphi}_L, O \underline{\varphi}_R). \\
\end{array}    
\end{equation}

\subsection{Renormalization on $\underline{\mathcal{D}}$}

Let $I=[a, b] \subset [0, 1]$ and let $1_I:[0,1] \rightarrow [a,b]$ be the affine map

\[
1_I(x)=|I|x+a = (b-a)x+a
\]
which has the inverse $1_I^{-1}:[a, b] \rightarrow [0,1]$ given by

\[
1_I^{-1}(x)= \frac{x-a}{|I|} = \frac{x-a}{b-a}.
\]
It is known that the zoom operator $\varsigma_I : \mathcal{C}^1([0,1]) \rightarrow \mathcal{C}^1([0,1])$ is defined by 

\begin{equation}
\varsigma_I \varphi (x) = 1_{\varphi (I)}^{-1} \circ \varphi \circ 1_{I}(x).
\end{equation}
Observe that the nonlinearity operator satisfies 

\[
N(\varsigma_I \varphi) = |I| \cdot N \varphi \circ 1_I.
\]
 Thus we define the zoom operator $Z_I: \mathcal{C}^1([0,1]) \rightarrow \mathcal{C}^1([0,1])$ acting on a nonlinearity by

\begin{equation}
\label{zoom operator law}
Z_I \eta (x) = |I| \cdot \eta \circ 1_I(x),     
\end{equation}
and if $\varphi$ is a $\mathcal C^2$ diffeomorphism we define $Z_I\varphi$ by

\begin{equation*}
\begin{array}{cclll}
     Z_I & : & \mbox{Diff}_{+}^{r}([0, 1]) & \rightarrow & \mathcal{C}^{r-2}([0,1]) \\
     & & \varphi & \mapsto & Z_I \varphi (x) = |I| \cdot \eta_{\varphi} \circ  1_I(x) \\
\end{array}    
\end{equation*}
where $\eta_{\varphi}=N\varphi.$ 

\smallskip

It will be convenient to introduce a different set of coordinates on the space of gap mappings.
We denote by $\Sigma$ the unit cube
\[
\Sigma = (0, 1)^3 = \{ (\alpha, \beta, b ) \in \mathbb{R}^3 \; | \; 0 < \alpha, \beta, b < 1 \},
\]
by $\mbox{Diff}_+^3([0,1])^2$ the set

\[
\{ (\varphi_L, \varphi_R) \; | \; \varphi_L, \; \varphi_R : [0, 1] \rightarrow [0, 1] \; \mbox{are orientation preserving } \mathcal{C}^3- \mbox{diffeomorphisms}\}
\]
and by

\[
\mathcal D' = \Sigma \times \mbox{Diff}_+^3([0, 1])^2.
\]


We define a change of coordinates from $\mathcal D'$ to $\mathcal{D}$ by:

\begin{equation}
\begin{array}{llccl}
    \Theta & : & \mathcal D' & \rightarrow & \mathcal{D}  \\
     & & (\alpha, \beta, b, \varphi_L, \varphi_R )& \mapsto & \Theta( \alpha, \beta, b, \varphi_L, \varphi_R )=\colon f \\
\end{array}    
\end{equation}
where $f:[b-1, b] \setminus \{0\} \rightarrow [b-1,b]$ is defined by

\begin{equation}
\label{definition of a gap map f in M}
f(x) = \left\{ \begin{array}{lcl}
f_L(x)    & ,& x \in [b-1,0) \\
f_R(x)  & ,& x \in (0,b]
\end{array}    \right.
\end{equation}
with 

\begin{equation}
\label{definition of f_L}
\begin{array}{llccl}
    f_L & : & I_{0,L}=[b-1,0]  & \rightarrow & T_{0,L}=[\alpha (b-1)+b, b]  \\
     & & x & \mapsto & f_L(x) = 1_{T_{0,L}} \circ \varphi_L \circ 1_{I_{0,L}}^{-1}(x) \\
\end{array}    
\end{equation}
and 

\begin{equation}
\label{definition of f_R} 
\begin{array}{llccl}
    f_R & : & I_{0,R}=[0, b]  & \rightarrow &  T_{0,R}=[b-1, \beta b + b-1] \\
     & & x & \mapsto & f_R(x) = 1_{T_{0,R}} \circ \varphi_R \circ 1_{I_{0,R}}^{-1}(x). \\
\end{array}    
\end{equation}
Note that $f_L$ and $f_R$ are differentiable and strictly increasing functions such that $0 < f_L '(x) \leq \nu < 1$, for all $x\in [b-1,0]$, and $0 < f_R '(x) \leq \nu < 1$, for all $x\in [0, b]$, where $\nu$ is a positive real number and less than $1$ depending on $f$, i.e. $\nu = \nu_f \in (0,1)$. The functions $\varphi_L$ and $\varphi_R$ are called the {\em diffeomorphic parts} of $f$. See Figure~\ref{fig: slopes and branches of f}.

\begin{figure}
\begin{center}
\begin{tikzpicture}
\path[draw, line width=0.3mm, top color=white!40, bottom color=white!40] (0, 0.0)  -- (8, 0.0)  -- (8, 8)  -- (0,8)  -- (0,0) --cycle;

\path[draw, line width=0.3mm,top color=white!40,bottom
color=white!40] (0,0)  --  (8,8) --cycle;

\path[draw, line width=0.3mm,top color=white!40,bottom
color=white!40] (6, 0)  --  (6, 8) --cycle;

\path[draw, line width=0.3mm,top color=white!40,bottom
color=white!40] (0, 6)  --  (8,6) --cycle;

\draw [cyan] plot [smooth] coordinates {(0,2.8) (4,5.4) (6,8)};

\draw [cyan] plot [smooth] coordinates {(6,0) (6.5,0.45) (7,0.8) (7.5,1) (8,1.45) };

\draw[black] (6,-0.5) node {$0$};
\draw[black] (0.3,-0.5) node {$b-1$};
\draw[black] (8,-0.5) node {$b$};

\draw[black] (-0.5,6) node {$0$};
\draw[black] (-0.7,0.1) node {$b-1$};
\draw[black] (-0.5,8) node {$b$};

\draw[black] (-1,2.8) node {$f_L(b-1)$};
\draw[black] (8.6,1.45) node {$f_R(b)$};

\draw[black] (-3,4.8) node {$\alpha = \displaystyle \frac{b-f_L(b-1)}{0-(b-1)}$};
\draw[black] (10.6,0.45) node {$\beta = \displaystyle \frac{f_R(b)-(b-1)}{b-0}$};

\draw [dashed] (0,2.8) -- (2.8,2.8);
\draw [dashed] (1.45,1.45) -- (8,1.45);
\draw [dashed,blue] (1.45,0) -- (1.45,1.45);
\draw [dashed,blue] (2.8,0) -- (2.8,2.8);

\draw[black] (6.6,1.1) node {$f_R$};
\draw[black] (2.5,5) node {$f_L$};

\draw[(-, ultra thick, blue] (1.45,0) -- (1.46,0);
\draw[-), ultra thick, blue] (2.7,0) -- (2.8,0);
\draw[blue] (1.46,0) -- (2.7,0);

\draw[blue] (2.1,-0.5) node {$G$};

\end{tikzpicture}
\end{center}
    \caption{Branches $f_L$ and $f_R$, slopes $\alpha$ and $\beta$ of a gap map $f$.}
    \label{fig: slopes and branches of f}
\end{figure}
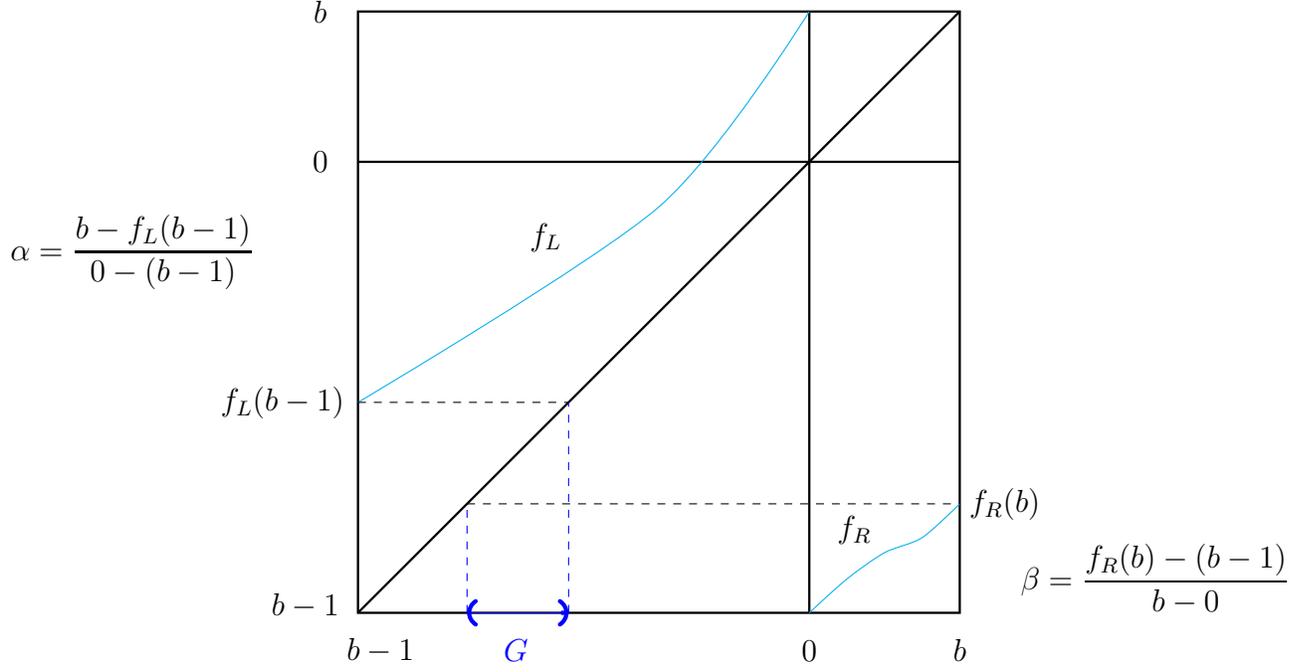

\begin{remark}
Depending on the properties of a gap mapping that we wish to emphasize, we can express a gap mapping $f$ in either coordinate system:
$f=(f_L, f_R, b)$ or $f=(\alpha,\beta, b,\varphi_L,\varphi_R),$ and we will move freely between the two coordinate systems. 
\end{remark}

Let $\mathcal D_0$ denote the set of once renormalizable gap mappings. 
If $f=(\alpha,\beta,b,\varphi_L,\varphi_R)\in\mathcal D_0$, we let $\tilde f=\mathcal R f=(\tilde \alpha,\tilde \beta, \tilde b,\tilde\varphi_L,\tilde\varphi_R)$ denote its renormalization. When $\sigma_f=-$, we have the following expressions for the coordinates of $\tilde f:$ 
\begin{equation}
\label{renormalized coordinates case -}
\begin{array}{ccl}
\tilde{\alpha} & = & \displaystyle  \frac{f_L^k \circ f_R \circ f_L(0_{k+1}^{+})-0_{k+2}^{-}}{0_{k+1}^{+}} \\
\tilde{\beta}  & = &  \displaystyle \frac{f_L^k \circ f_R(0_{k+2}^-)-0_{k+1}^+}{0_{k+2}^-} \\
\tilde{b}      & = &  \displaystyle \frac{0_{k+2}^-}{|[0_{k+1}^+,0_{k+2}^-]|}\\
& & \\
\tilde{\varphi}_L & = & \displaystyle \varsigma_{[0_{k+1}^+,0]} \tilde{f}_L , \; \mbox{ with } \tilde{f}_L=f_L^k \circ f_R \circ f_L \\
& & \\
\tilde{\varphi}_R & = & \displaystyle \varsigma_{[0, 0_{k+2}^-]} \tilde{f}_R, \; \mbox{ with } \tilde{f}_R=f_L^k \circ f_R.
\end{array}
\end{equation}
We have 
similar expressions when $\sigma_f=+,$
which we omit.

To express 
$\tilde{\underline{f}}\in\underline{\mathcal{D}},$
we write $\tilde{\underline{f}}=(\tilde \alpha,\tilde \beta, \tilde b,\tilde{\underline{\varphi}}_L,\tilde{\underline{\varphi}}_R)$, where 
$\tilde \alpha,\tilde \beta$ and $\tilde b$ are as in formula (\ref{renormalized coordinates case -}),
and 
$ \tilde{\underline{\varphi}}_L$ and $\tilde{\underline{\varphi}}_R,$
are defined by:
$$\displaystyle
\tilde{\underline{\varphi}}_L=
\zeta_{U_{k+2}}\underline f_L \oplus \zeta_{U_{k+1}}\underline f_L
\oplus\dots \zeta_{U_{2}}\underline f_L\oplus
\zeta_{U_{1}}\underline f_R\oplus\zeta_{U_{0}}\underline f_L,
\mbox{ and}
$$
$$
\displaystyle
\tilde{\underline{\varphi}}_R=
\zeta_{V_{k+1}}\underline f_L \oplus \zeta_{V_{k}}\underline f_L
\oplus\dots \zeta_{V_{1}}\underline f_L\oplus
\zeta_{V_{0}}\underline f_R,
$$
where $\underline{f}_L$ and $\underline{f}_R$ are decompositions over a singleton timeset,
$U_0=(0^+_{k+1},0),$ $U_i=f^i(U_0)$
for $0<i\leq k+2$,
$V_0=(0,0^-_{k+2}),$ and $V_i=f^i(V_0)$
for $0<i\leq k+1.$ One immediately sees that after composing the decomposed mappings we obtain $\tilde f.$

As we will use the structure of Banach space in $\mbox{Diff}_{+}^{3}([0, 1])$ given by the nonlinearity operator we need the expressions for the coordinates functions $\tilde{\varphi}_L$ and $\tilde{\varphi}_R$ in terms of the zoom operator. Note that the coordinates $\tilde{\alpha}$, $\tilde{\beta}$ and $\tilde{b}$ remain the same as in (\ref{renormalized coordinates case -}) 
since they are not affected by the zoom operator. In order to obtain these coordinate functions we need to apply the zoom operator to each branch of the first return map $R$ on the interval $I'= [0_{k+1}^{+}, 0_{k+2}^{-}]$, \label{page: I'}
in case $\sigma_f=-$, or on the interval $I'= [0_{k+2}^{+}, 0_{k+1}^{-}]$, in case $\sigma_f=+$. Then, when $\sigma_f=-$, we obtain
\begin{equation}
\label{tilde eta_L and eta_R}
\begin{array}{ccl}
\tilde{\eta}_L     & = & = Z_{[0_{k+1}^{+}, 0]} \eta_{\tilde{f}_L} = |0_{k+1}^{+}| \cdot \eta_{\tilde{f}_L} \circ 1_{[0_{k+1}^{+}, 0]}^{-1} = 
|0_{k+1}^{+}| \cdot N(f_L^k \circ f_R \circ f_L) \circ 1_{[0_{k+1}^{+}, 0]}^{-1}\\
& & \\
\tilde{\eta}_R     & = & = Z_{[0, 0_{k+2}^{-}]} \eta_{\tilde{f}_R} = 
|0_{k+2}^{-}| \cdot \eta_{\tilde{f}_R} \circ 1_{[0, 0_{k+2}^{-}]}^{-1} =
|0_{k+2}^{-}| \cdot N(f_L^k \circ f_R) \circ 1_{[0, 0_{k+2}^{-}]}^{-1}, \\
\end{array}    
\end{equation}
The formulas when $\sigma=+$ are similar, and to save space we do not include them.

\begin{remark}
We would like to stress that throughout the remainder of this paper we will make use of the Banach space structure on $\mbox{\em Diff}_{+}^{3}([0, 1])$ given by 
its identification with $\mathcal C^1([0,1])$ via the nonlinearity operator. 
\end{remark}

\section{The derivative of the renormalization operator}\label{sec:derivative}
In this section we will estimate the derivative of the renormalization operator acting on an absorbing set under renormalization in the decomposition space of dissipative gap mappings. A little care is needed since the operator is not differentiable.

Recall that $\mathcal D_0\subset\mathcal C^3$, is the 
set of once renormalizable gap dissipative gap mappings. Then
$\mathcal R:\mathcal D_0\to\mathcal C^2$ is differentiable, and the derivative $D\mathcal R_f:\mathcal C^3 \to \mathcal C^2$ extends to a bounded operator
$D\mathcal R_f:\mathcal C^2\to\mathcal C^2,$
which depends continuously on $f\in\mathcal C^3.$ In \cite{MP}, 
$\mathcal R$ is called {\em jump-out differentiable}.

\label{derivative of renormalization operator}
If $\underline{f}=(\alpha, \beta, b, \underline{\varphi}_L, \underline{\varphi}_R) \in \underline{\mathcal D}_0$ the derivative of $\underline{\mathcal{R}}_{\underline{f}}$, $D\underline{\mathcal{R}}_{\underline{f}}$ is a matrix of the form

\begin{equation}
\label{derivative matrix}
D \underline{\mathcal{R}}_{\underline{f}} = \left( \begin{array}{cc}
A_{\underline{f}}     &  B_{\underline{f}} \\
C_{\underline{f}}     &  D_{\underline{f}}
\end{array}     \right)
\end{equation}
where

\begin{enumerate}
    \item[.] $A_{\underline{f}}: \mathbb{R}^3 \rightarrow \mathbb{R}^3$
    \item[.] $B_{\underline{f}}: X \times X \rightarrow \mathbb{R}^3$
    \item[.] $C_{\underline{f}}: \mathbb{R}^3 \rightarrow X \times X$
    \item[.] $D_{\underline{f}}: X \times X \rightarrow X \times X.$
\end{enumerate}
We estimate $A_{\underline{f}}$ in Lemma~\ref{lemma partial derivatives of abb tilde respec to abb},
$B_{\underline{f}}$ in Lemma~\ref{lemma entries of B_f matrix},
$C_{\underline{f}}$ in Lemma~\ref{lemma for Cmatrix}
and $D_{\underline{f}}$ in Lemma~\ref{lemma entries of D matrix}.

In order to estimate the entries of matrices $A_{\underline{f}}$, $B_{\underline{f}}$, $C_{\underline{f}}$ and $D_{\underline{f}}$ we will make use of the partial derivative operator $\partial$. The main properties of $\partial$ are presented in the next lemma.

\begin{lemma}\cite[Lemma 9.4]{MarcoWinclker}
\label{lemma partial operator}
The following equations hold whenever they make sense:

\begin{equation}
\label{derivative of f o g}
\partial (f \circ g)(x) =   \partial f(g(x))+f'(g(x)) \partial g(x)  ,
\end{equation}
\begin{equation}
\label{derivative of fn+1}
\partial (f^{n+1})(x) =   \sum_{i=0}^{n}Df^{n-i}(f^{i+1}(x)) \partial f(f^{i}(x)), 
\end{equation}
\begin{equation}
\partial (f^{-1})(x) = - \frac{\partial f(f^{-1}(x))}{f'(f^{-1}(x))} ,
\end{equation}
\begin{equation}
\partial (f \cdot g)(x) =   \partial f(x)g(x) +f(x)\partial g(x) , 
\end{equation}
\begin{equation}
\partial (f/g)(x) =  \frac{\partial f(x)g(x)-f(x) \partial g(x)}{(g(x))^2}.  
\end{equation}
\end{lemma}

From now on we will make use of the notation

\[
g(x) \asymp y
\]
to mean that there exists a positive constant $K < \infty$ not depending on $g$ such that $K^{-1}y \leq g(x) \leq K y$, for all $x$ in the domain of $g$.

\smallskip

Recall that the inverse of the nonlinearity operator $N: \mbox{Diff}_+^3([0,1]) \rightarrow \mathcal{C}^1([0,1])$ is 
given by
\begin{equation}
\label{expression of phi_eta}    
\varphi (x) = \varphi_{\eta}(x)=N^{-1} \eta (x) = \frac{\int_0^x e^{\int_0^s \eta (t)dt}ds}{\int_0^1 e^{\int_0^s \eta (t)dt}ds},
\end{equation}
where $\eta \in \mathcal{C}^1([0,1]).$

\begin{lemma}
\label{lemma d phi d eta}
Let $x \in [0, 1]$. The evaluation operator $E: \mbox{\em Diff}_{+}^2([0,1]) = \mathcal{C}^0([0,1]) \rightarrow \mathbb{R}$

\[
E: \eta \mapsto \varphi_{\eta}(x)
\]
is differentiable with derivative $\displaystyle \frac{\partial \varphi (x)}{\partial \eta}: \mathcal{C}^0([0,1]) \rightarrow \mathbb{R}$ given by

\begin{equation}
\label{formula for d phi d eta}
\displaystyle \frac{\partial \varphi (x)}{\partial \eta}(\Delta \eta ) = \left( \displaystyle \frac{\int_{0}^{x} \left[ \int_{0}^{s} \Delta \eta \right] e^{\int_{0}^{s} \eta }ds}{ \int_{0}^{x}e^{\int_{0}^{s}\eta} ds}
- \displaystyle \frac{\int_{0}^{1} \left[ \int_{0}^{s} \Delta \eta \right] e^{\int_{0}^{s} \eta }ds}{ \int_{0}^{1}e^{\int_{0}^{s}\eta} ds}\right) \varphi (x).
\end{equation}
There exists $\varepsilon_0>0$ so that for all $\varepsilon\in(0,\varepsilon_0),$ if $\| D^2\varphi \|_{\mathcal Cˆ0} < \varepsilon,$ 
we have that
\begin{equation}
\label{estimate for d phi d eta}
\frac{1}{8}\min\{ \varphi (x), 1-\varphi (x)\}\leq \left| \displaystyle \frac{\partial \varphi (x)}{\partial \eta}  \right|  \leq 2 \min\{ \varphi (x), 1-\varphi (x)\}.   
\end{equation}
\end{lemma}

\begin{proof}
In order to prove that the evaluation operator $E$ is (Fr\'echet) differentiable and obtain the formula (\ref{formula for d phi d eta}) we just need to use the Gateaux variation to look for a candidate $T$ for its derivative, i.e.

\begin{equation}
\begin{array}{ccl}
    \displaystyle T(\eta) \Delta \eta & = & \displaystyle \frac{d}{dt}E(\eta + t \Delta \eta ) \left|_{t=0}  \right.. \\
     & & \\
\end{array}    
\end{equation}
Since this calculation is not difficult we left it to the reader. Now we will prove the estimate (\ref{estimate for d phi d eta}). Using techniques of integration we obtain

\begin{equation}
\label{integral of product}
\displaystyle \int_{0}^{x} \left[ \int_{0}^{s} \Delta \eta \right] e^{\int_{0}^{s} \eta }ds = \displaystyle \big( \int_{0}^{x} \Delta \eta \big) \cdot \int_{0}^{x} e^{\int_{0}^{t} \eta }ds - \int_{0}^{x} \left[ \Delta \eta \cdot \int_{0}^{s} e^{\int_{0}^{t} \eta } \right] ds.
\end{equation}
From (\ref{integral of product}), (\ref{formula for d phi d eta}) and (\ref{expression of phi_eta}) and some manipulations we obtain

\begin{equation}
\label{norm of d phi d eta}    
\left| \displaystyle \frac{\partial \varphi (x)}{\partial \eta}(\Delta \eta ) \right| = \varphi (x) \cdot \int_{x}^{1} \Delta \eta ds - \varphi (x) \cdot \int_{0}^{1} \Delta \eta \cdot \varphi(s) ds + \int_{0}^{x} \Delta \eta \cdot \varphi(s) ds. 
\end{equation}
From the definition of the norm
\[
\left| \displaystyle \frac{\partial \varphi (x)}{\partial \eta}(\Delta \eta ) \right| = \sup_{||\Delta \eta|| = 1} \left| \displaystyle \frac{\partial \varphi (x)}{\partial \eta} \right|,
\]
we can substitute $\Delta \eta = 1$ at (\ref{norm of d phi d eta}) and obtain

\[
\left| \displaystyle \frac{\partial \varphi (x)}{\partial \eta}(\Delta \eta ) \right| = \varphi (x) \cdot (1-x) - \varphi (x) \cdot \int_{0}^{1} \varphi(s) ds + \int_{0}^{x} \varphi(s) ds.
\]
Using the fact that in deep renormalization the map $\varphi$ is close to identity, i.e. $\|\varphi (x) - x\|_{\mathcal Cˆ0}$ is small, so we get

\begin{equation}
\begin{array}{ccl}
   \left| \displaystyle \frac{\partial \varphi (x)}{\partial \eta}(\Delta \eta ) \right|  & \asymp & x \cdot (1-x) - x \cdot \int_{0}^{1} s ds + \int_{0}^{x} s ds \\
   & & \\
     & = & \displaystyle \frac{x}{2} \cdot (1-x). \\ 
\end{array}    
\end{equation}
Since 
\[
\displaystyle T_{\frac{1}{4}}(x) \leq \frac{x}{2}(1-x) \leq T_{2}(x)
\]
for all $x \in [0, 1]$, where $T_c(x)$ is the tent map family $T_c:[0,1]\to[0,1]$, defined by
$$T_c(x)=\left\{
\begin{array}{cc}
    cx & \mbox{ for } x\in[0,1/2]\\
     -cx+c &  \mbox{ for } x\in(1/2,1].
\end{array}\right.
$$

The result follows.
\end{proof}

\begin{cor}\cite[Corollary 8.17]{MP}
\label{corollary d fgh dg}
Let $\psi^+, \psi^- \in \mbox{\em Diff}_{+}^2([0, 1])$ and $x \in [0, 1]$. The evaluation operator 

\begin{equation}
\begin{array}{ccccl}
     E^{\psi_+, \psi^-} & : & \mbox{\em Diff}_{+}^2([0, 1]) = \mathcal{C}^0([0,1]) & \rightarrow  & \mathbb{R} \\
     &  & \eta & \mapsto & E^{\psi^+, \psi^-} (\eta) = \psi^+ \circ \varphi_{\eta} \circ \psi^- (x) \\ 
\end{array}    
\end{equation}
is differentiable with derivative $\displaystyle \frac{\partial \big( \psi^+ \circ \varphi_{\eta} \circ \psi^- (x) \big) }{\partial \eta } : \mathcal{C}^0([0,1]) \rightarrow \mathbb{R}$ given by

\begin{equation}
\displaystyle      \frac{\partial \big( \psi^+ \circ \varphi_{\eta} \circ \psi^- (x) \big) }{\partial \eta } \big( \Delta \eta \big) = D \psi^+ (\varphi_{\eta} \circ \psi^- (x)) \cdot \frac{\partial \varphi_{\eta} (\psi^-(x))}{\partial \eta } \big( \Delta \eta \big).
\end{equation}

\end{cor}

The next result follows from a straightforward calculation, and its proof is left to the reader.

\begin{lemma}
\label{lemma derivatives of f_L and f_R}
The branches $f_L$ and $f_R$ of $f$ defined in (\ref{definition of a gap map f in M}) are differentiable and their partial derivatives are given by

\begin{equation}
\label{derivatives formulas of f_l and f_R}
\begin{array}{ccl}
     \displaystyle \frac{\partial f_L}{\partial \alpha }(x)   & = &  (1-b) \cdot \displaystyle \big[ \varphi_L \big(  \frac{x-b+1}{1-b} \big)  - 1 \big], \hspace{1cm} \displaystyle \frac{\partial f_L}{\partial \beta }(x) =0 \\
     && \\
     \displaystyle \frac{\partial f_L}{\partial b }(x)   & = &  1+ \alpha \cdot \displaystyle \big[1- \varphi_L \big(  \frac{x-b+1}{1-b} \big)  \big] + \frac{\alpha x}{1-b} D \varphi_L \big(  \frac{x-b+1}{1-b} \big), \\
     && \\
     \displaystyle \frac{\partial f_L}{\partial \eta_L }(x)   & = & \displaystyle |T_{0, L}| \cdot \frac{\partial \varphi_L (1_{I_{0, L}}^{-1}(x))}{\partial \eta_L}, \hspace{1cm} \displaystyle \frac{\partial f_L}{\partial \eta_R }(x) = 0\\
     && \\
     \displaystyle \frac{\partial f_R}{\partial \alpha }(x)   & = &  0, \hspace{1cm} \displaystyle \frac{\partial f_R}{\partial \beta }(x) = b \varphi_R \big( \frac{x}{b} \big) \\
     && \\
     \displaystyle \frac{\partial f_R}{\partial b }(x)   & = &  1+ \beta \cdot \displaystyle \varphi_R \big(  \frac{x}{b} \big)  - \frac{\beta x}{b} D \varphi_R \big(  \frac{x}{b} \big), \\
     && \\
      \displaystyle \frac{\partial f_R}{\partial \eta_L }(x)   & = & 0, \hspace{1cm} \displaystyle \frac{\partial f_R}{\partial \eta_R }(x) = \displaystyle |T_{0, R}| \cdot \frac{\partial \varphi_R (1_{I_{0, R}}^{-1}(x))}{\partial \eta_R}. \\
\end{array}
\end{equation}
Furthermore, all these partial derivatives are bounded.
\end{lemma}

Let $f=(f_L, f_R, b) \in \mathcal{D}$ be a renormalizable dissipative gap map. The boundaries of the the interval $I'=[0_{k+1}^{+}, 0_{k+2}^{-}]$, for $\sigma_f=-$, and $I'=[0_{k+2}^{+}, 0_{k+1}^{-}]$ for $\sigma_f=+$, can be interpreted as evaluation operators, that is

\begin{equation}
\label{boundary evaluation operator} 
\begin{array}{cccll}
 E    & : & M & \rightarrow & \mathbb{R} \\
     & & (\alpha, \beta, b, \varphi_L, \varphi_R)&  \mapsto & 0_{j}^{\pm}
\end{array}
\end{equation}
where $j \in \{ k+1, k+2 \}$ depending on the sign of $f$. For convenience we will call $0_{j}^{\pm}$ as {\it boundary operators}. The next result give us some properties about the boundary operators. 
\begin{lemma}
\label{Lemma with formulas}
The boundary operators $0_{j}^{\pm}$ are differentiable, and the partial derivatives $\displaystyle \frac{\partial 0_{j}^{\pm}}{\partial *}$ are bounded, where $* \in \{ \alpha, \beta, b, \eta_L, \eta_R \}$, and $j \in \{ k+1, k+2 \}$, depending on the sign of $f$.
\end{lemma}

\begin{proof}
Consider the boundary operators $0_{k+2}^{-}$ and $0_{k+1}^{+}$, which are explicitly given by

\[
0_{k+1}^{+} = f_L^k(b-1), \hspace{1cm} \mbox{and} \hspace{1cm} 0_{k+2}^{-} = f_L^k \circ f_R (b),
\]
when $\sigma_f=-$, and where $f_L = 1_{T_{0,L}} \circ \varphi_L \circ 1_{I_{0,L}}^{-1}$ and $f_R = 1_{T_{0,R}} \circ \varphi_R \circ 1_{I_{0,R}}^{-1}$. Using (\ref{derivative of fn+1}) and taking $* \in \{ \alpha, \beta, b, \eta_L, \eta_R \}$ we get 

\begin{equation}
\label{leftt boundary}
\begin{array}{ccl}
    \displaystyle \frac{\partial }{\partial *} \big( 0_{k+1}^{+} \big) & =  \displaystyle \frac{\partial}{\partial *} \big(  f_L^k (b-1) \big) =& \displaystyle \sum_{i=0}^{k-1}Df_L^{k-1-i}(f_L^{i+1} (b-1)) \cdot  \frac{\partial f_L}{\partial *} (f_L^i (b-1)),\\
\end{array}    
\end{equation}
and

\begin{equation}
\label{right boundary}
\begin{array}{ccl}
    \displaystyle \frac{\partial }{\partial *} \big( 0_{k+2}^{-} \big) & =  \displaystyle \frac{\partial}{\partial *} \big(  f_L^k \circ f_R (b) \big) =& \displaystyle \sum_{i=0}^{k-1}Df_L^{k-1-i}(f_L^{i+1} \circ f_R(b)) \cdot  \frac{\partial f_L}{\partial *} (f_L^i \circ f_R(b))\\
    && \\
    & & \displaystyle + Df_L^k \circ f_R (b) \cdot \frac{\partial f_R}{\partial *}(b).  \\
\end{array}    
\end{equation}
Using the fact that $0 < f'(x) \leq \nu < 1$, for all $x \in [b-1,b] \setminus \{ 0 \}$, and Lemma~\ref{lemma derivatives of f_L and f_R} we get that $\displaystyle \frac{\partial }{\partial *} \big( 0_{k+2}^{-} \big)$ and $\displaystyle \frac{\partial }{\partial *} \big( 0_{k+1}^{+} \big)$ are bounded. With similar arguments and reasoning we prove that the other boundary operators have bounded partial derivatives.
\end{proof}

\subsection{The $A_{\underline{f}}$ matrix}

\begin{equation}
\label{partial derivatives A for sigma -}
A_{\underline{f}} = \left( \begin{array}{ccc}
\displaystyle \frac{\partial \tilde{\alpha}}{\partial \alpha }     & \displaystyle \frac{\partial \tilde{\alpha}}{\partial \beta } & \displaystyle \frac{\partial \tilde{\alpha}}{\partial b } \\
& & \\
 \displaystyle \frac{\partial \tilde{\beta}}{\partial \alpha }     & \displaystyle \frac{\partial \tilde{\beta}}{\partial \beta } & \displaystyle \frac{\partial \tilde{\beta}}{\partial b } \\
 & & \\
 \displaystyle \frac{\partial \tilde{b}}{\partial \alpha }     & \displaystyle \frac{\partial \tilde{b}}{\partial \beta } & \displaystyle \frac{\partial \tilde{b}}{\partial b } \\
\end{array}           \right),
\end{equation}

All the entries of matrix $A_{\underline{f}}$ can be calculated explicitly by using Lemma~\ref{lemma partial operator}. In order to clarify the calculations we will compute some of them in the next lemma.

\begin{lemma}
\label{lemma partial derivatives of abb tilde respec to abb}
Let $ \underline{f}=(\alpha,\beta,b,\underline{\varphi}_R,
\underline{\varphi}_L) \in \underline{\mathcal D}_0$. The map
\begin{equation}
\begin{array}{ccl}
(0, 1)^3 \ni (\alpha, \beta, b)     & \mapsto & (\tilde{\alpha}, \tilde{\beta}, \tilde{b}) \in (0, 1)^3\\
\end{array}    
\end{equation}
is differentiable. Furthermore, 
for any $\varepsilon>0, K>0$
if $\underline g\in\underline{\mathcal{D}}_0$
is infinitely renormalizable, there exists 
$n_0\in\mathbb N,$ so that if $n\geq n_0$ and $\underline f=\underline{\mathcal{R}}^n \underline g,$ then
the partial derivatives $\displaystyle \Big|\frac{\partial }{\partial \alpha } \tilde{\alpha}\Big|$, $\displaystyle \Big| \frac{\partial }{\partial \beta } \tilde{\alpha}\Big|$, $\displaystyle \Big| \frac{\partial }{\partial b } \tilde{\alpha}\Big|$, $\displaystyle \Big| \frac{\partial }{\partial \alpha } \tilde{\beta}\Big|$, $\displaystyle \Big|\frac{\partial }{\partial \beta } \tilde{\beta}\Big|$ and $\displaystyle \Big|\frac{\partial }{\partial b } \tilde{\beta}\Big|$ are all bounded from above by $\varepsilon$, and the partial derivatives $\displaystyle \Big|\frac{\partial }{\partial \alpha } \tilde{b}\Big|$, $\displaystyle \Big| \frac{\partial }{\partial \beta } \tilde{b}\Big|$ and $\displaystyle\Big| \frac{\partial }{\partial b } \tilde{b}\Big|$ are bounded from below by $K$. In particular $\displaystyle\Big| \frac{\partial }{\partial b } \tilde{b}\Big|\asymp\frac{1}{|I'|}.$ 
(See page~\pageref{page: I'} for the definition of $I'$.)
\end{lemma}

\begin{proof}
We will prove this lemma in the case where $\sigma_f=-$. The case $\sigma_f=+$ is similar and we will leave it to the reader. 
From (\ref{renormalized coordinates case -}) we obtain the partial derivatives

\begin{equation}
\label{partial derivatives of abb tilde respec to abb}
\begin{array}{ccl}
\displaystyle \frac{\partial }{\partial *} \tilde{\alpha }    & = & \displaystyle \frac{1}{(0_{k+1}^{+})^2} \cdot \left\{ 0_{k+1}^{+} \cdot \displaystyle \frac{\partial }{\partial *} \big( f_L^k \circ f_R \circ f_L (0_{k+1}^{+}) \big) - 0_{k+1}^{+} \cdot  \displaystyle \frac{\partial }{\partial *} \big( 0_{k+2}^{-} \big) \right. \\
&& \\
&& \hspace{1.5cm} \left. - \displaystyle \big[ f_L^k \circ f_R \circ f_L (0_{k+1}^{+}) - 0_{k+2}^{-} \big] \cdot  \displaystyle \frac{\partial }{\partial *} \big( 0_{k+1}^{+} \big)   \right\} \\
&& \\
\displaystyle \frac{\partial }{\partial *} \tilde{\beta }    & = & \displaystyle \frac{1}{(0_{k+2}^{-})^2} \cdot \left\{ 0_{k+2}^{-} \cdot \displaystyle \frac{\partial }{\partial *} \big( f_L^k \circ f_R (0_{k+2}^{-}) \big) - 0_{k+2}^{-} \cdot  \displaystyle \frac{\partial }{\partial *} \big( 0_{k+1}^{+} \big) \right. \\
&& \\
&& \hspace{1.5cm} \left. - \displaystyle \big[ f_L^k \circ f_R (0_{k+2}^{-}) - 0_{k+1}^{+} \big] \cdot  \displaystyle \frac{\partial }{\partial *} \big( 0_{k+2}^{-} \big)   \right\} \\
&& \\
\displaystyle \frac{\partial }{\partial *} \tilde{b}    & = & (1-\tilde{b}) \cdot |I'|^{-1} \cdot \displaystyle \frac{\partial }{\partial *} \big( f_L^k \circ f_R(b) \big) + 
|I'|^{-1} \cdot \tilde{b} \cdot \displaystyle \frac{\partial }{\partial *} \big( f_L^k (b-1) \big),\\
&& \\
\end{array}    
\end{equation}
where $* \in \{ \alpha, \beta, b \}$. Let us start to deal with the first line of $A_{\underline{f}}$, that is, with the partial derivatives

\[
\displaystyle \frac{\partial \tilde{\alpha}}{\partial *}
\]
where $* \in \{ \alpha, \beta, b\}$. Taking $* = \alpha$ we obtain

\begin{equation}
\label{derivative of alpha tilde resp to alpa}
\begin{array}{ccl}
\displaystyle \frac{\partial }{\partial \alpha} \tilde{\alpha }    & = & \displaystyle \frac{1}{(0_{k+1}^{+})^2} \cdot \left\{ 0_{k+1}^{+} \cdot \displaystyle \frac{\partial }{\partial \alpha} \big( f_L^k \circ f_R \circ f_L (0_{k+1}^{+}) \big) - 0_{k+1}^{+} \cdot  \displaystyle \frac{\partial }{\partial \alpha} \big( 0_{k+2}^{-} \big) \right. \\
&& \\
&& \hspace{1.5cm} \left. - \displaystyle \big[ f_L^k \circ f_R \circ f_L (0_{k+1}^{+}) - 0_{k+2}^{-} \big] \cdot  \displaystyle \frac{\partial }{\partial \alpha} \big( 0_{k+1}^{+} \big)   \right\}. \\
&& \\
\end{array}    
\end{equation}
From (\ref{derivative of f o g}) and using the fact that $f_R$ does not depend on $\alpha$ we have

\begin{equation}
\label{application of derivative operator to f_L^k f_R f_L}
\begin{array}{ccl}
\displaystyle \frac{\partial }{\partial \alpha} \big( f_L^k \circ f_R \circ f_L (0_{k+1}^{+}) \big) & = & \displaystyle \frac{\partial }{\partial \alpha} \big( f_L^k \big) \circ f_R \circ f_L (0_{k+1}^{+}) \\
&& \\
&& + D \big( f_L^k  \circ f_R \big) \circ f_L(0_{k+1}^{+}) \cdot 
\displaystyle \frac{\partial }{\partial \alpha} \big( f_L(0_{k+1}^{+}) \big)
\\
\end{array}    
\end{equation}
Since $0_{k+1}^{+}=f_L^k(b-1)$ we can apply (\ref{derivative of fn+1}) and get

\begin{equation}
\label{derivative of f_L(0_{k+1}^{+}) resp to alpha}
\displaystyle \frac{\partial }{\partial \alpha} \big( f_L(0_{k+1}^{+}) \big) = \displaystyle \frac{\partial }{\partial \alpha} \big( f_L^{k+1}(b-1) \big) = \displaystyle \sum_{i=0}^{k} Df_L^{k-i}(f_L^{i+1}(b-1)) \cdot \displaystyle \frac{\partial f_L}{\partial \alpha} (f_L^i(b-1)).    
\end{equation}
Since $0_{k+2}^{-} = f_L^k \circ f_R(b)$ by applying the Mean Value Theorem to the difference $f_L^k \circ f_R \circ f_L (0_{k+1}^{+}) - 0_{k+2}^{-} $ we obtain a point $\xi \in (f_L(0_{k+1}^{+}),b)$ such that

\begin{equation}
\label{xi point}
\displaystyle f_L^k \circ f_R \circ f_L (0_{k+1}^{+}) - 0_{k+2}^{-} = 
f_L^k \circ f_R \circ f_L (0_{k+1}^{+})-f_L^k \circ f_R (b) =
D \big( f_L^k \circ f_R \big)(\xi) \cdot \big[ f_L(0_{k+1}^{+})- b \big].
\end{equation}
Since $b=f_L(0^-)$ by applying the Mean Value Theorem once more we obtain another point $\zeta \in (0_{k+1}^{+}, 0)$ such that
\begin{equation}
\label{zeta point}
f_L(0_{k+1}^{+})- b = f_L(0_{k+1}^{+})- f_L(0^-) = D f_L (\zeta ) \cdot 0_{k+1}^{+}.    
\end{equation}
Substituting (\ref{zeta point}), (\ref{xi point}) and (\ref{application of derivative operator to f_L^k f_R f_L})  into (\ref{derivative of alpha tilde resp to alpa}) and after some manipulations we get

\begin{equation}
\label{derivative of alpha tilde resp to alpa step 2}
\begin{array}{ccl}
\displaystyle \frac{\partial }{\partial \alpha} \tilde{\alpha }    & = & \displaystyle \frac{1}{(0_{k+1}^{+})} \cdot \left\{  \displaystyle \frac{\partial }{\partial \alpha} \big( f_L^k \big) \circ f_R \circ f_L (0_{k+1}^{+})  - \displaystyle \frac{\partial }{\partial \alpha} \big( f_L^k \big) \circ f_R(b)\right. \\
&& \\
&& \hspace{1.5cm} \left.+D(f_L^k \circ f_R) \circ f_L(0_{k+1}^{+}) \cdot \displaystyle \frac{\partial }{\partial \alpha} \big( f_L(0_{k+1}^{+}) \big) \right. \\
&& \\
&& \hspace{1.5cm} \left. - \displaystyle \big[ D \big( f_L^k \circ f_R \big)(\xi) \cdot D f_L (\zeta )  \big] \cdot  \displaystyle \frac{\partial }{\partial \alpha} \big( 0_{k+1}^{+} \big)   \right\}. \\
&& \\
\end{array}    
\end{equation}

By (\ref{derivative of fn+1}) we obtain

\begin{equation}
\begin{array}{l}
    \displaystyle \frac{\partial }{\partial \alpha} \big( f_L^k \big) \circ f_R \circ f_L (0_{k+1}^{+}) - \frac{\partial }{\partial \alpha} \big( f_L^k \big) \circ f_R (b)  \\
    \\
= \displaystyle \sum_{i=0}^{k-1}Df_L^{k-1-i} \big( f_L^{i+1} \circ f_R \circ f_L(0_{k+1}^{+})\big) \cdot \displaystyle \frac{\partial f_L}{\partial \alpha}\big( f_L^i \circ f_R \circ f_L(0_{k+1}^{+})\big) \\
\\
- \displaystyle \sum_{i=0}^{k-1}Df_L^{k-1-i} \big( f_L^{i+1} \circ f_R(b)\big) \cdot \displaystyle \frac{\partial f_L}{\partial \alpha} \big( f_L^i \circ f_R (b)\big). \\
\end{array}    
\end{equation}
From Lemma~\ref{lemma derivatives of f_L and f_R} we know that  $\displaystyle  \frac{\partial f_L}{\partial \alpha }(x)$ is bounded, then putting 

\[
C_1 = \max_{0 \leq i < k } \left\{   \displaystyle \big| \frac{\partial f_L}{\partial \alpha }(f_L^i \circ f_R \circ f_L(0_{k+1}^{+})) \big|, \big|  \frac{\partial f_L}{\partial \alpha }(f_L^i \circ f_R (b)) \big| \right\},
\]
we obtain

\begin{equation}
\begin{array}{l}
    \displaystyle \big| \frac{\partial }{\partial \alpha} \big( f_L^k \big) \circ f_R \circ f_L (0_{k+1}^{+}) - \frac{\partial }{\partial \alpha} \big( f_L^k \big) \circ f_R (b) \big| \\
    \\
\leq C_1 \cdot \displaystyle \sum_{i=0}^{k-1} \big| Df_L^{k-1-i} \big( f_L^{i+1} \circ f_R \circ f_L(0_{k+1}^{+})\big) - Df_L^{k-1-i} \big( f_L^{i+1} \circ f_R(b)\big) \big|. \\
\end{array}    
\end{equation}
Applying the Mean Value Theorem twice we obtain a point $\xi_i \in (f_{L}^{i+1} \circ f_R \circ f_L(0_{k+1}^{+}), f_{L}^{i+1} \circ f_R(b))$, and a point $\theta_i \in (f_L(0_{k+1}^{+}), b)$ such that

\begin{equation}
\label{xi_i and theta_i points}
 \begin{array}{l}
\big| Df_L^{k-1-i} \big( f_L^{i+1} \circ f_R \circ f_L(0_{k+1}^{+})\big) - Df_L^{k-1-i} \big( f_L^{i+1} \circ f_R(b)\big) \big|      \\
\\
= \big| D^2 f_L^{k-1-i}(\xi_i)\big| \cdot \big| D(f_L^{i+1} \circ f_R) (\theta_i)\big| \cdot \big| Df_L(\zeta)\big| \cdot |0_{k+1}^{+}|.       
 \end{array}   
\end{equation}
From this we obtain

\begin{equation}
\label{using xi_i and theta_i points first time}
\begin{array}{l}
 \displaystyle \big| \frac{\partial }{\partial \alpha} \big( f_L^k \big) \circ f_R \circ f_L (0_{k+1}^{+}) - \frac{\partial }{\partial \alpha} \big( f_L^k \big) \circ f_R (b) \big| \\
    \\
\leq C_1 \cdot |0_{k+1}^{+}| \cdot \displaystyle \sum_{i=0}^{k-1} \big| D^2 f_L^{k-1-i}(\xi_i)\big| \cdot \big| D(f_L^{i+1} \circ f_R) (\theta_i)\big| \cdot \big| Df_L(\zeta)\big| \\
\\
= C_1 \cdot  |0_{k+1}^{+}| \cdot \big| Df_L(\zeta)\big| \cdot \displaystyle \sum_{i=0}^{k-1} \big| D^2 f_L^{k-1-i}(\xi_i)\big| \cdot \big| Df_L^{i} \circ f_L \circ f_R (\theta_i)\big| \cdot \big| Df_L \circ f_R (\theta_i) \big| \cdot \big| Df_R(\theta_i) \big|. \\
\end{array}    
\end{equation}

For the other difference in (\ref{derivative of alpha tilde resp to alpa step 2}) we start by observing that $\displaystyle \frac{\partial }{\partial \alpha} \big( f_L(0_{k+1}^{+})\big)$ and $\displaystyle \frac{\partial }{\partial \alpha} \big( 0_{k+1}^{+} \big)$ are either simultaneously positive or negative. Furthermore, from Lemma~\ref{Lemma with formulas} we have that $\displaystyle \frac{\partial }{\partial \alpha} \big( 0_{k+1}^{+} \big)$ is bounded, and
arguing similarly, we have that $\displaystyle \frac{\partial }{\partial \alpha} \big( f_L(0_{k+1}^{+})\big)$ is also bounded. Thus there exists a constant $\displaystyle C_2>0$ such that
 
 \begin{equation}
 \label{second difference in alpha tilde to alpha}
\begin{array}{l}
\big| \displaystyle D(f_L^k \circ f_R) \circ f_L(0_{k+1}^{+}) \cdot \frac{\partial }{\partial \alpha} \big( f_L(0_{k+1}^{+})\big) 
 - \displaystyle  \big[ D \big( f_L^k \circ f_R \big)(\xi) \cdot D f_L (\zeta ) \big] \cdot  \displaystyle \frac{\partial }{\partial \alpha} \big( 0_{k+1}^{+} \big) \big| \\  
 \\
 \leq C_2 \cdot \big| D(f_L^k \circ f_R) \circ f_L(0_{k+1}^{+}) - D(f_L^k \circ f_R) (\xi) \big| \\
 \\
 \leq C_2 \cdot \big| D^2(f_L^k \circ f_R)(w)\big| \cdot \big| Df_L (\zeta) \big| \cdot |0_{k+1}^{+}| \\
\end{array}     
 \end{equation}
where $w \in (f_L(0_{k+1}^{+}), \xi)$ is a point given by the Mean Value Theorem. 

Substituting (\ref{using xi_i and theta_i points first time}) and (\ref{second difference in alpha tilde to alpha}) into (\ref{derivative of alpha tilde resp to alpa step 2}) we obtain

\begin{equation}
\label{estimate of alpha tilde resp to alpha} 
\begin{array}{ccl}
\displaystyle \big|  \frac{\partial }{\partial \alpha} \tilde{\alpha}\big|     & \leq &   C_1 \cdot \big| Df_L(\zeta)\big| \cdot \displaystyle \sum_{i=0}^{k-1} \big| D^2 f_L^{k-1-i}(\xi_i)\big| \cdot \big| Df_L^{i} \circ f_L \circ f_R (\theta_i)\big| \cdot \big| Df_L \circ f_R (\theta_i) \big| \cdot \big| Df_R(\theta_i) \big| \\
     & & \\
     & & + C_2 \cdot \big| D^2(f_L^k \circ f_R)(w)\big| \cdot \big| Df_L (\zeta) \big|.\\
\end{array}
\end{equation}

Since the first and second derivatives of $f$ goes to zero when the level of renormalization goes to infinity we conclude that $
\displaystyle \big| \frac{\partial }{\partial \alpha} \tilde{\alpha} \big| \longrightarrow 0,$
when the level of renormalization goes to infinity. With same arguments and reasoning we can prove that $\displaystyle \big| \frac{\partial }{\partial \beta} \tilde{\alpha} \big|$, $\displaystyle \big| \frac{\partial }{\partial b} \tilde{\alpha} \big|$, $\displaystyle \big| \frac{\partial }{\partial \alpha} \tilde{\beta} \big|$, $\displaystyle \big| \frac{\partial }{\partial \beta} \tilde{\beta} \big|$ and $\displaystyle \big| \frac{\partial }{\partial b} \tilde{\beta} \big|$ all tend to zero as the level of renormalization tends to infinity.

Now we will prove that $\displaystyle \big| \frac{\partial \tilde{b} }{\partial b}  \big|$ is big. From (\ref{partial derivatives of abb tilde respec to abb}) we have

\begin{equation}
\label{eqn:57}
\begin{array}{ccl}
\displaystyle \big| \frac{\partial \tilde{b}}{\partial b}  \big|   &  = & \displaystyle \frac{1}{|I'|^2} \cdot \left\{  0_{k+2}^{-} \cdot \frac{\partial }{\partial b} \big( 0_{k+1}^{+} \big) - 0_{k+1}^{+} \cdot \frac{\partial }{\partial b} \big( 0_{k+2}^{-} \big) \right\} \\
     &  & \\
     & \geq & \displaystyle \frac{1}{|I'|} \cdot \mbox{min} \left\{ \frac{\partial }{\partial b} \big( 0_{k+1}^{+} \big), \frac{\partial }{\partial b} \big( 0_{k+2}^{-} \big) \right\} \\
     & & \\
     & \geq & \displaystyle \frac{1}{|I'|} \cdot \mbox{min} \left\{  \frac{\partial f_L}{\partial b} \big( f_L^{k-1}(b-1) \big), \frac{\partial f_L}{\partial b} \big( f_L^{k-1} \circ f_R(b) \big) \right\} \\
     & & \\
\end{array}    
\end{equation}
which is big since the size of $I'$ goes to infinity when the level of renormalization is deeper, and from Lemma~\ref{lemma derivatives of f_L and f_R} we get that $\displaystyle \frac{\partial f_L}{\partial b} \big( f_L^{k-1} \circ f_R(b) \big)$ and $\displaystyle \frac{\partial f_L}{\partial b} \big( f_L^{k-1}(b-1) \big)$ are both greater than a positive constant $c>1/3$. With the same arguments we prove that $\displaystyle \big| \frac{\partial \tilde{b} }{\partial \alpha}  \big|$ and $\displaystyle \big| \frac{\partial \tilde{b} }{\partial \beta}  \big|$ are big.
\end{proof}

\begin{remark}
\label{remark similar calculations}
We note that all the calculations used to get $\frac{\partial \tilde{\alpha}}{\partial \alpha }(x)$ in the above proof of Lemma~\ref{lemma partial derivatives of abb tilde respec to abb} we can use to get the others partial derivatives $\frac{\partial \tilde{\alpha}}{\partial \beta }(x)$, $\frac{\partial \tilde{\alpha}}{\partial b }(x)$, $\frac{\partial \tilde{\alpha}}{\partial \eta_L} $ and $\frac{\partial \tilde{\alpha}}{\partial \eta_R }(x)$, just observing that in each case the constants will depend on the specific partial derivative we are calculating, that is, in the calculation of $\frac{\partial \tilde{\alpha}}{\partial \eta_L }(x)$ the constants $C_1$ and $C_2$ will depend on $\frac{\partial f_L}{\partial \eta_L }$.
\end{remark}

\subsection{The $B_{\underline{f}}$ matrix}

\begin{equation}
\label{partial derivatives B for sigma -}
B_{\underline{f}} = \left( \begin{array}{cc}
\displaystyle \frac{\partial \tilde{\alpha}}{\partial \eta_L }     & \displaystyle \frac{\partial \tilde{\alpha}}{\partial \eta_R } \\
& \\
\displaystyle \frac{\partial \tilde{\beta}}{\partial \eta_L } & 
\displaystyle \frac{\partial \tilde{\beta}}{\partial \eta_R } \\
 & \\
 \displaystyle \frac{\partial \tilde{b}}{\partial \eta_L }     & \displaystyle \frac{\partial \tilde{b}}{\partial \eta_R } \\ 
\end{array}           \right),
\end{equation}

\begin{lemma}
\label{lemma entries of B_f matrix}
Let $\underline{f} \in \underline{\mathcal D}_0$. The maps
\begin{equation}
\begin{array}{ccl}
{\mathcal C}^1([0, 1]) \ni \eta_L     & \mapsto &  (\tilde{\alpha}, \tilde{\beta}, \tilde{b}) \in (0, 1)^3\\
&& \\
 {\mathcal C}^1([0, 1]) \ni \eta_R   & \mapsto &  (\tilde{\alpha}, \tilde{\beta}, \tilde{b}) \in (0, 1)^3
\end{array}    
\end{equation}
are differentiable. 
Moreover, for any $\varepsilon>0$, if $\underline g\in\underline{\mathcal{D}}$
is infinitely renormalizable, and $\underline f=\underline{\mathcal{R}}\underline g$, then there exists $n_0\in\mathbb N$ so that for $n\geq n_0$ we have that 
$\displaystyle \Big| \frac{\partial \tilde{\alpha}}{\partial \eta_L }\Big|,  \displaystyle \Big|\frac{\partial \tilde{\alpha}}{\partial \eta_R }\Big|,
\displaystyle\Big| \frac{\partial \tilde{\beta}}{\partial \eta_L }\Big|,
\displaystyle \Big|\frac{\partial \tilde{\beta}}{\partial \eta_R }\Big|<\varepsilon$, $\displaystyle \Big|\frac{\partial \tilde{b}}{\partial \eta_R }\Big|=0,$ and $\displaystyle \Big| \frac{\partial \tilde{b}}{\partial \eta_L }\Big|\asymp\frac{b}{|I'|},$ where $I'$ is as defined on page~\pageref{page: I'}.
\end{lemma}
\begin{proof}
From (\ref{renormalized coordinates case -}) the expressions of the partial derivatives of $\tilde{\alpha}$, $\tilde{\beta}$ and $\tilde{b}$ are given by

\begin{equation}
\label{partial derivatives of bab resp lr}
\begin{array}{ccl}
\displaystyle \frac{\partial }{\partial *} \tilde{\alpha }    & = & \displaystyle \frac{1}{(0_{k+1}^{+})^2} \cdot \left\{ 0_{k+1}^{+} \cdot \displaystyle \frac{\partial }{\partial *} \big( f_L^k \circ f_R \circ f_L (0_{k+1}^{+}) \big) - 0_{k+1}^{+} \cdot  \displaystyle \frac{\partial }{\partial *} \big( 0_{k+2}^{-} \big) \right. \\
&& \\
&& \hspace{1.5cm} \left. - \displaystyle \big[ f_L^k \circ f_R \circ f_L (0_{k+1}^{+}) - 0_{k+2}^{-} \big] \cdot  \displaystyle \frac{\partial }{\partial *} \big( 0_{k+1}^{+} \big)   \right\} \\
&& \\
\displaystyle \frac{\partial }{\partial *} \tilde{\beta }    & = & \displaystyle \frac{1}{(0_{k+2}^{-})^2} \cdot \left\{ 0_{k+2}^{-} \cdot \displaystyle \frac{\partial }{\partial *} \big( f_L^k \circ f_R (0_{k+2}^{-}) \big) - 0_{k+2}^{-} \cdot  \displaystyle \frac{\partial }{\partial *} \big( 0_{k+1}^{+} \big) \right. \\
&& \\
&& \hspace{1.5cm} \left. - \displaystyle \big[ f_L^k \circ f_R (0_{k+2}^{-}) - 0_{k+1}^{+} \big] \cdot  \displaystyle \frac{\partial }{\partial *} \big( 0_{k+2}^{-} \big)   \right\} \\
&& \\
\displaystyle \frac{\partial }{\partial *} \tilde{b}    & = & (1-\tilde{b}) \cdot |I'|^{-1} \cdot \displaystyle \frac{\partial }{\partial *} \big( f_L^k \circ f_R(b) \big) + 
|I'|^{-1} \cdot \tilde{b} \cdot \displaystyle \frac{\partial }{\partial *} \big( f_L^k (b-1) \big)\\
&& \\
\end{array}    
\end{equation}
where $* \in \{ \eta_L, \eta_R \}$.  With similar arguments used in the proof of Lemma~\ref{lemma partial derivatives of abb tilde respec to abb} we can prove that 
\[
\displaystyle \frac{\partial \tilde{\alpha}}{\partial \eta_L}, \;\; \frac{\partial \tilde{\alpha}}{\partial \eta_R}, \;\; \frac{\partial \tilde{\beta}}{\partial \eta_L}, \;\; \frac{\partial \tilde{\beta}}{\partial \eta_R}
\]
are as small as we want.

Now let us estimate
\[
\displaystyle \frac{\partial }{\partial \eta_L} \tilde{b} \hspace{1cm} \mbox{and} \hspace{1cm}  \frac{\partial }{\partial \eta_R} \tilde{b}. 
\]
Observe that at deep levels of renormalization the diffeomorphic parts $\varphi_L$ and $\varphi_R$ are very close to the identity function, so we can assume that

\[
\varphi_L(x) = x + o(\epsilon), \hspace{1cm} \varphi_R(x) = x + o(\epsilon)
\]
where $\epsilon>0$ is arbitrarily small. With some manipulations, we get from (\ref{partial derivatives of bab resp lr})

\begin{equation}
\label{d b tilde d eta_L}
    \begin{array}{ccl}
     \displaystyle \frac{\partial }{\partial \eta_L} \tilde{b}    &  = & \displaystyle \frac{1}{|I'|^2} \cdot \left\{ 0_{k+2}^{-} \cdot \frac{\partial }{\partial \eta_L} \big( 0_{k+1}^{+} \big) -0_{k+1}^{+} \cdot \frac{\partial }{\partial \eta_L} \big( 0_{k+2}^{-}\big)    \right\}.\\
    \end{array}
\end{equation}

Let us analyze each term inside the braces separately. Since $$0_{k+1}^{+} = f_L^k(b-1) = f_L (f_L^{k-1}(b-1))\mbox{ and }f_L = 1_{T_{0,L}} \circ \varphi_L \circ 1_{I_{0,L}}^{-1}$$ we obtain

\begin{equation}
\label{d 0k+1 d eta_L}
\begin{array}{ccl}
\displaystyle \frac{\partial }{\partial \eta_L} \big( 0_{k+1}^{+} \big)     &  = & \displaystyle \frac{\partial }{\partial \eta_L} \big(  f_L (f_L^{k-1}(b-1)) \big)\\
     & & \\
     & = &  D 1_{T_{0,L}} \circ \big( \varphi_L \circ  1_{I_{0,L}}^{-1} \circ f_L^{k-1}(b-1) \big) \cdot \displaystyle \frac{\partial }{\partial \eta_L} \big(  \varphi_L \circ  1_{I_{0,L}}^{-1} (f_L^{k-1}(b-1)) \big)\\
     & & \\
     & \asymp & |T_{0,L}| \cdot \min\{ \varphi_L \circ  1_{I_{0,L}}^{-1} (f_L^{k-1}(b-1)), 1- \varphi_L \circ  1_{I_{0,L}}^{-1} (f_L^{k-1}(b-1)) \} \\
     & & \\ 
      & \asymp & |T_{0,L}| \cdot \min\{ 1_{I_{0,L}}^{-1} (f_L^{k-1}(b-1)), 1- 1_{I_{0,L}}^{-1} (f_L^{k-1}(b-1)) \} \\
     & & \\ 
     & = & |T_{0,L}| \cdot \big(1-1_{I_{0,L}}^{-1} (f_L^{k-1}(b-1)) \big). \\
\end{array}    
\end{equation}
By using analogous arguments we get

\begin{equation}
\label{d 0k+2 d eta_L}
\begin{array}{ccl}
     \displaystyle \frac{\partial }{\partial \eta_L} \big( 0_{k+2}^{-} \big)     &  = & \displaystyle \frac{\partial }{\partial \eta_L} \big(  f_L (f_L^{k-1}(f_R(b))) \big)\\
     & & \\
 & \asymp & |T_{0,L}| \cdot \big(1-1_{I_{0,L}}^{-1} (f_L^{k-1}(f_R(b))) \big). \\
\end{array}    
\end{equation}
Substituting (\ref{d 0k+2 d eta_L}) and (\ref{d 0k+1 d eta_L}) into (\ref{d b tilde d eta_L}) we get

\begin{equation}
\label{d b tilde d eta_L second step}
\begin{array}{ccl}
    \displaystyle \frac{\partial }{\partial \eta_L} \tilde{b}    &  \asymp & \displaystyle \frac{|T_{0,L}|}{|I'|^2} \cdot \left\{ 0_{k+2}^{-} \cdot \big(1-1_{I_{0,L}}^{-1} (f_L^{k-1}(b-1)) \big) -0_{k+1}^{+} \cdot \big(1-1_{I_{0,L}}^{-1} (f_L^{k-1}(f_R(b))) \big)   \right\}.\\ & & \\
    & = & \displaystyle \frac{|T_{0,L}|}{|I'|^2} \cdot \left[ 0_{k+2}^{-} - 0_{k+1}^{+}  \right]  + \displaystyle \frac{|T_{0,L}|}{|I'|^2} \cdot 0_{k+1}^{+} \cdot 1_{I_{0,L}}^{-1} (f_L^{k-1}(f_R(b))) \\
    & & \\
    & & - \displaystyle \frac{ |T_{0,L}| }{|I'|^2} \cdot 0_{k+2}^{-} \cdot 1_{I_{0,L}}^{-1} (f_L^{k-1}(b-1)).  \\
\end{array}    
\end{equation}
Since the size of the renormalization interval $I'$ goes to zero when the level of renormalization goes to infinity we can assume that $b-0_{k+2}^{-} \asymp b$ and then we have

\begin{equation}
|0_{k+1}^{-}|=0-f_L^{k-1}(f_R(b)) = \frac{b-0_{k+2}^{-}}{Df_L(c_1)} \asymp \frac{b}{Df_L(c_1)} \asymp b \cdot \frac{|I_{0,L}|}{|T_{0,L}|},
\end{equation}
where we use the assumption that 

\begin{equation}
\label{approx to Df_L}
\left. \begin{array}{ccl}
     f_L & = & 1_{T_{0,L}} \circ \varphi_L \circ 1_{I_{0,L}}^{-1}  \\
     & & \\
     \varphi_L & \approx & \mbox{identity function} \\ 
\end{array} \right\}    \Rightarrow Df_L = \frac{|T_{0,L}|}{|I_{0,L}|} \cdot D \varphi_L \asymp \frac{|T_{0,L}|}{|I_{0,L}|}.
\end{equation}
By using the approximation (\ref{approx to Df_L}) we have

\begin{equation}
\label{approx to 0k}
|0_{k}^{+}|=0-f_L^{k-1}(b-1) = \frac{b-0_{k+1}^{+}}{Df_L(c_2)} \asymp  (b-0_{k+1}^{+}) \cdot \frac{|I_{0,L}|}{|T_{0,L}|}.
\end{equation}
Using (\ref{approx to 0k}), (\ref{approx to Df_L}) and the definition of the affine map $1_{I_{0,L}}^{-1}$ by (\ref{d b tilde d eta_L second step}) we obtain

\begin{equation}
\label{eqn:72}
    \begin{array}{ccl}
    \displaystyle \frac{\partial }{\partial \eta_L} \tilde{b}     &  \asymp &  \displaystyle \frac{-b}{|I'|} + \frac{0_{k+1}^{+} \cdot 0_{k+2}^{-}}{|I'|^2}. \\
    \end{array}
\end{equation}
Since $I'= [0_{k+1}^{+}, 0_{k+2}^{-}]$, and $|I'| \leq \alpha \cdot \beta \cdot b$ for all $k \geq 1$ we can conclude that $\displaystyle \frac{0_{k+1}^{+} \cdot 0_{k+2}^{-}}{|I'|^2}$
is bounded and thus
$\displaystyle  \big| \frac{\partial }{\partial \eta_L} \tilde{b} \big|
 \asymp \frac{-b}{|I'|}.$
For the derivative of $\tilde{b}$ with respect to $\eta_R$ we start by noting that $0_{k+1}^{+} = f_L^{k}(b-1)$ does not depend on $\eta_R$. Hence, with similar arguments used to get (\ref{d b tilde d eta_L}) we obtain

\begin{equation}
\label{d b tilde d eta_R}
    \begin{array}{ccl}
     \displaystyle \frac{\partial }{\partial \eta_R} \tilde{b}    &  = & \displaystyle \frac{1}{|I'|^2} \cdot \left\{ 0_{k+2}^{-} \cdot \frac{\partial }{\partial \eta_R} \big( 0_{k+1}^{+} \big) -0_{k+1}^{+} \cdot \frac{\partial }{\partial \eta_R} \big( 0_{k+2}^{-}\big)    \right\} \\
     & & \\
     & = & \displaystyle  \frac{-0_{k+1}^{+}}{|I'|^2} \cdot Df_L^k \circ f_R (b) \cdot \frac{\partial }{\partial \eta_R} 
     \big( f_R(b) \big). \\
    \end{array}
\end{equation}
Since $\displaystyle f_R = 1_{T_{o,R}} \circ \varphi_R \circ 1_{I_{0,R}}^{-1}$ and the point $1_{I_{0,R}}^{-1}(b)$ is always fixed by any $\varphi_R \in \mbox{Diff}_{+}^{3}[0,1]$ we obtain

\[
\displaystyle \frac{\partial }{\partial \eta_R} \big( \varphi_R \circ 1_{I_{0,R}}^{-1}(b) \big) =0
\]
and then

\[
\displaystyle \frac{\partial }{\partial \eta_R} 
    \big( f_R(b) \big) = D1_{T_{0,R}} \circ \big( \varphi_R \circ 1_{I_{0,R}}^{-1}(b) \big)  \cdot \frac{\partial }{\partial \eta_R} \big( \varphi_R \circ 1_{I_{0,R}}^{-1}(b) \big) = 0
\]
which implies in 
\[
\displaystyle \frac{\partial }{\partial \eta_R} \tilde{b}=0
\]
as desired.
\end{proof}

\subsection{The $C_{\underline{f}}$ matrix}

\begin{equation}
\label{partial derivatives C for sigma -}
C_{\underline{f}} = \left( \begin{array}{ccc}
\displaystyle \frac{\partial \tilde{\eta}_L}{\partial \alpha }     & \displaystyle \frac{\partial \tilde{\eta}_L}{\partial \beta } & \displaystyle \frac{\partial \tilde{\eta}_L}{\partial b } \\
& & \\
 \displaystyle \frac{\partial \tilde{\eta}_R}{\partial \alpha }     & \displaystyle \frac{\partial \tilde{\eta}_R}{\partial \beta } & \displaystyle \frac{\partial \tilde{\eta}_R}{\partial b }
\end{array}           \right),
\end{equation}

\begin{lemma}
\label{lemma for Cmatrix}
Let $ \underline{f} \in \underline{\mathcal D}_0 $. The maps

\begin{equation}
\begin{array}{ccl}
    (0,1)^3 \ni (\alpha, \beta, b) & \mapsto &  \tilde{\eta}_L \in {\mathcal C}^{1}([0, 1]) \\
    & & \\
{(0, 1)^3} \ni (\alpha, \beta, b) & \mapsto &  \tilde{\eta}_R \in  {\mathcal C}^{1}([0, 1]) \\ 
\end{array}    
\end{equation}
are differentiable and the partial derivatives are bounded. Furthermore, 
for any $\varepsilon>0$, if $\underline g\in\underline{\mathcal{D}}_0$ is an infinitely renormalizable mapping, there exists $N\in\mathbb n_0$ so that if $n\geq n_0$ and $\underline f=\underline{\mathcal{R}}^n \underline g$, 
we have that $
\displaystyle \Big| \frac{\partial \tilde{\eta}_L}{\partial \beta}\Big|$ and $\displaystyle\Big| \frac{\partial \tilde{\eta}_R}{\partial \beta}\Big|<\varepsilon
$, when $\sigma_f=-$,
and when $\sigma_f=+$ we have that
$
\displaystyle \Big|\frac{\partial \tilde{\eta}_L}{\partial \alpha}\Big|$
and $\displaystyle\Big|\frac{\partial \tilde{\eta}_R}{\partial \alpha}\Big|<\varepsilon.$
\end{lemma}

We will require some preliminary results before proving this lemma.
For the next calculations we deal only with the case $\sigma_f=-$, since case $\sigma_f=+$ is analogous. From (\ref{tilde eta_L and eta_R}) the partial derivatives of $\tilde{\eta}_L$ with respect to $\alpha$, $\beta$ and $b$ are given by

\begin{equation}
\begin{array}{ccl}
    \displaystyle \frac{\partial \tilde{\eta}_L}{\partial \alpha} & = & 
    \displaystyle \frac{\partial }{\partial \alpha} \left( Z_{[0_{k+1}^+,0]} \eta_{\tilde{f}_L} \right) \\
    && \\
    & = & \displaystyle \frac{\partial }{\partial  0_{k+1}^+} \left( Z_{[0_{k+1}^+,0]} \eta_{\tilde{f}_L} \right) \cdot \displaystyle \frac{\partial }{\partial \alpha} \left( 0_{k+1}^+ \right)
      + \displaystyle \frac{\partial }{\partial \eta_{\tilde{f}_L}} \left( Z_{[0_{k+1}^+,0]} \eta_{\tilde{f}_L} \right) \cdot \displaystyle \frac{\partial }{\partial \alpha} \left( \eta_{\tilde{f}_L} \right)
\end{array}
\end{equation}

\begin{equation}
\begin{array}{ccl}
    \displaystyle \frac{\partial \tilde{\eta}_L}{\partial \beta} & = & 
    \displaystyle \frac{\partial }{\partial \beta} \left( Z_{[0_{k+1}^+,0]} \eta_{\tilde{f}_L} \right) \\
    && \\
    & = & \displaystyle \frac{\partial }{\partial  0_{k+1}^+} \left( Z_{[0_{k+1}^+,0]} \eta_{\tilde{f}_L} \right) \cdot \displaystyle \frac{\partial }{\partial \beta} \left( 0_{k+1}^+ \right)
      + \displaystyle \frac{\partial }{\partial \eta_{\tilde{f}_L}} \left( Z_{[0_{k+1}^+,0]} \eta_{\tilde{f}_L} \right) \cdot \displaystyle \frac{\partial }{\partial \beta} \left( \eta_{\tilde{f}_L} \right)
\end{array}
\end{equation}

\begin{equation}
\begin{array}{ccl}
    \displaystyle \frac{\partial \tilde{\eta}_L}{\partial b} & = & 
    \displaystyle \frac{\partial }{\partial b} \left( Z_{[0_{k+1}^+,0]} \eta_{\tilde{f}_L} \right) \\
    && \\
    & = & \displaystyle \frac{\partial }{\partial  0_{k+1}^+} \left( Z_{[0_{k+1}^+,0]} \eta_{\tilde{f}_L} \right) \cdot \displaystyle \frac{\partial }{\partial b} \left( 0_{k+1}^+ \right)
      + \displaystyle \frac{\partial }{\partial \eta_{\tilde{f}_L}} \left( Z_{[0_{k+1}^+,0]} \eta_{\tilde{f}_L} \right) \cdot \displaystyle \frac{\partial }{\partial b} \left( \eta_{\tilde{f}_L} \right)
\end{array}
\end{equation}
We have similar expressions for the partial derivatives of $\tilde{\eta}_R$ with respect to $\alpha, \beta$ and $b$; however, we omit them at this point. 

\smallskip

In order to prove that all the six entries of $C_{\underline{f}}$ matrix are bounded we need to analyze the terms

\[
\displaystyle \frac{\partial }{\partial  0_{k+1}^+} \left( Z_{[0_{k+1}^+,0]} \eta_{\tilde{f}_L} \right), \;\;  
\displaystyle \frac{\partial }{\partial \eta_{\tilde{f}_L}} \left( Z_{[0_{k+1}^+,0]} \eta_{\tilde{f}_L} \right), \;\; \frac{\partial }{\partial *} \left( 0_{k+1}^+ \right), \;\; \displaystyle \frac{\partial }{\partial *} \left( \eta_{\tilde{f}_L} \right) 
\]
with $* \in \{ \alpha, \beta, b\}$ for $\tilde{\eta}_L$, and the corresponding ones for $\tilde{\eta}_R$. This analysis will be done in the following lemmas.

\begin{lemma}\cite[Lemma 8.20]{MP}
\label{lemma Zoom derivatives}
Let $\varphi \in \mbox{\em Diff}_{+}^{3}([0,1])$. The zoom curve $Z:[0,1]^2 \ni (a,b) \mapsto Z_{[a,b]} \varphi \in \mbox{ \em Diff}_{+}^2([0,1])$ is differentiable with partial derivatives given by

\begin{equation}
\label{derivative of left boundary for Z}
\begin{array}{ccl}
    \displaystyle \frac{\partial Z_{[a,b]} \varphi }{\partial a} & = & 
    (b-a)(1-x)D \eta ((b-a)x+a)-\eta ((b-a)x+a)\\
    && \\
     \displaystyle \frac{\partial Z_{[a,b]} \varphi }{\partial b} & = & 
     (b-a)xD \eta ((b-a)x+a)+\eta ((b-a)x+a).
\end{array}
\end{equation}
The norms are bounded by
\begin{equation}
\label{norms of derivatives of Z} 
\big| \displaystyle \frac{\partial Z_{[a,b]} \varphi }{\partial a} \big|_{2}, \big| \displaystyle \frac{\partial Z_{[a,b]} \varphi }{\partial b} \big|_{2} \leq 2|\varphi|_3.
\end{equation}
Furthermore, by considering a fixed interval $I \subset [0,1]$, the zoom operator
\begin{equation}
\begin{array}{cccll}
    Z_I & : & \mathcal{C}^1([0,1])& \rightarrow & \mathcal{C}^1([0,1]) \\
        &   & \varphi & \mapsto & Z_I \varphi, 
\end{array}    
\end{equation}
where $Z_I \varphi (x)$ is defined in (\ref{zoom operator law}), is differentiable with respect to $\eta$ and its derivative is given by

\[
\displaystyle \frac{\partial }{\partial \varphi } \left( Z_I \varphi \right) (\Delta g) = |I| \cdot \Delta g \circ 1_I,
\]
and its norm is given by
\[
\big| \big|\displaystyle \frac{\partial }{\partial \varphi } \left( Z_I \varphi \right) \big| \big| = |I|.
\]
\end{lemma}

Since the nonlinearity of affine maps is zero it is not difficult to check that the nonlinearity of the branches $f_L = 1_{T_{0,L}} \circ \varphi_L \circ 1_{I_{0,L}}^{-1}$ and $f_R = 1_{T_{0,R}} \circ \varphi_R \circ 1_{I_{0,R}}^{-1}$ are 

\begin{equation}
\label{nonlinearity of branches f_L and f_R}
\begin{array}{ccl}
     N f_L & = & \displaystyle \frac{1}{|I_{0,L}|} \cdot N \varphi_L \circ 1_{I_{0,L}}^{-1}, \\
     && \\
     N f_R & = & \displaystyle \frac{1}{|I_{0,R}|} \cdot N \varphi_R \circ 1_{I_{0,R}}^{-1}.
\end{array}
\end{equation}
Hence we note that $N f_L$ depends only on $b$ and $\varphi_L$ while $Nf_R$ depends only on $b$ and $\varphi_R$. Thus, we can derive $N f_L$ with respect to $b$ and $\varphi_L$, and we can derive $N f_R$ with respect to $b$ and $\varphi_R$. This is treated in the next result.

\begin{lemma}
\label{derivatives of compositions with nonlinearity}
Let $\underline{f} \in \underline{\mathcal D}_0$ and let $g$ be a $\mathcal{C}^1$ function. If the partial derivatives of $g$ with respect to $\alpha$, $\beta$ and $b$ are bounded, then, 
whenever the expressions make sense,
the compositions $Nf_L \circ g (x)$ and $Nf_R \circ g(x)$ are differentiable and the corresponding partial derivatives are bounded.
\end{lemma}

\begin{proof}
From (\ref{nonlinearity of branches f_L and f_R}) and Lemma \ref{lemma partial operator} we get

\begin{equation}
\begin{array}{ccl}
   \displaystyle \frac{\partial }{\partial b} \big[ Nf_L \circ g(x)  \big]  & 
   = & \displaystyle \frac{- \frac{\partial }{\partial b}|I_{0,L}|}{|I_{0,L}|^2} \cdot N \varphi_L \circ 1_{I_{0,L}}^{-1} \circ g(x) \\
    && \\
    && + \displaystyle \frac{1}{|I_{0,L}|} \cdot D N \varphi_L \circ 1_{I_{0,L}}^{-1} \circ g(x) \cdot \displaystyle \frac{\partial }{\partial b} \big( 1_{I_{0,L}}^{-1} \circ g(x)\big). \\
\end{array}
\end{equation}
For $Nf_R \circ g(x)$ we have a similar expression for its derivative with respect to $b$ just changing $I_{0,L}$ by $I_{0,R}$ and $\varphi_L$ by $\varphi_R$. The other partial derivatives are

\begin{equation}
\begin{array}{ccl}
   \displaystyle \frac{\partial }{\partial *} \big[ Nf_L \circ g(x)  \big] & 
   = & \displaystyle DNf_L \circ g(x) \cdot
     \displaystyle \frac{\partial }{\partial *}  g(x), \\
     && \\
      \displaystyle \frac{\partial }{\partial *} \big[ Nf_R \circ g(x)   \big] & 
   = & \displaystyle DNf_R \circ g(x) \cdot
     \displaystyle \frac{\partial }{\partial * }  g(x), \\
\end{array}
\end{equation}
where $* \in \{ \alpha, \beta \}$. Since our gap mappings $f=(f_L, f_R, b)$ have Schwarzian derivative $Sf$ and nonlinearity $Nf$ bounded, by the formula of the Schwarzian derivative

\[
\displaystyle Sf = D(Nf) - \frac{1}{2}(Nf)^2,
\]
we obtain that the derivative of the nonlinearity $D(Nf)$ is bounded. Using the hypothesis that the function $g$ has bounded partial derivatives the result follows as desired.
\end{proof}

The next result is about a property that the nonlinearity operator satisfies and which we will need. A proof for it can be found in \cite{MarcoWinclker}.

\begin{lemma}[The chain rule for the nonlinearity operator.]
\label{Lemma Chain rule for nonlinearity}
If $\phi, \psi \in \mathcal{D}^2$ then
\begin{equation}
\label{formula chain rule for nonlinearity}
N(\psi \circ \phi) = N \psi \circ \phi \cdot D \phi + N \phi.
\end{equation}
\end{lemma}

An immediately consequence of Lemma~\ref{Lemma Chain rule for nonlinearity} is the following result.
\begin{cor}
\label{cor:bdd N}
The operators 
\begin{equation}
\begin{array}{ccl}
     (\alpha, \beta, b) & \mapsto  & \eta_{\tilde{f}_L}:=N(\tilde{f}_L)=N(f_L^k \circ f_R \circ f_L)\\
     && \\
     (\alpha, \beta, b) & \mapsto  & \eta_{\tilde{f}_R}:=N(\tilde{f}_R)=N(f_L^k \circ f_R)
\end{array}    
\end{equation}
are differentiable. Furthermore, their partial derivatives are bounded.
\end{cor}

\begin{proof}
From Lemma~\ref{Lemma Chain rule for nonlinearity} we obtain
\begin{equation}
\begin{array}{ccl}
\eta_{\tilde{f}_L} & = & N(\tilde{f}_L) = N(f_L^k \circ f_R \circ f_L)  \\
   & = & \displaystyle \sum_{i=1}^{k}Nf_L \big( f_L^{k-i} \circ f_R \circ f_L \big) \cdot Df_{L}^{k-i} \circ f_R \circ f_L \cdot D(f_R \circ f_L) \\  
   && \\
   && + Nf_R \circ f_L \cdot Df_L + Nf_L \\
&& \\
\eta_{\tilde{f}_R} & = & N(\tilde{f}_R) = N(f_L^k \circ f_R)  \\
   & = & \displaystyle \sum_{i=1}^{k}Nf_L \big( f_L^{k-i} \circ f_R \big) \cdot Df_{L}^{k-i} \circ f_R \cdot Df_R +Nf_R. \\  
\end{array}    
\end{equation}
Taking $* \in \{ \alpha, \beta, b \}$ we have

\begin{equation}
\label{derivative of eta_{tilde{f}_L}}
\begin{array}{ccl}
\displaystyle \frac{\partial }{\partial *} \eta_{\tilde{f}_L} & = & \displaystyle \sum_{i=1}^{k} \frac{\partial }{\partial *} \big[ Nf_L \big( f_L^{k-i} \circ f_R \circ f_L \big) \cdot Df_{L}^{k-i} \circ f_R \circ f_L \cdot D(f_R \circ f_L) \big] \\  
   && \\
   && + \displaystyle \frac{\partial }{\partial *} \big[ Nf_R \circ f_L \cdot Df_L \big] + \frac{\partial }{\partial *} \big[ Nf_L \big] \\
\end{array}    
\end{equation}
and

\begin{equation}
\label{derivative of eta_{tilde{f}_R}}
\begin{array}{ccl}
\displaystyle \frac{\partial }{\partial *} \eta_{\tilde{f}_R} & = & \displaystyle \sum_{i=1}^{k} \frac{\partial }{\partial *} \big[ Nf_L \big( f_L^{k-i} \circ f_R \big) \cdot Df_{L}^{k-i} \circ f_R \cdot Df_R \big] + \displaystyle \frac{\partial }{\partial *} \big[ Nf_R \big]. \\
\end{array}    
\end{equation}
Since $f_L = 1_{T_{0,L}} \circ \varphi_L \circ 1_{I_{0,L}}^{-1}$, $f_R = 1_{T_{0,R}} \circ \varphi_R \circ 1_{I_{0,R}}^{-1}$, $T_{0,L}= [\alpha (b-1)+b, b]$, $T_{0,R}= [b-1, \beta b + b-1]$, $I_{0,L}= [b-1, 0]$, and $I_{0,R}= [0, b]$ we obtain

\[
Df_L = \displaystyle \frac{|T_{0,L}|}{|I_{0, L}|} \cdot D \varphi_L = \alpha \cdot D \varphi_L \hspace{1cm} \mbox{and} \hspace{1cm}
Df_R = \displaystyle \frac{|T_{0,R}|}{|I_{0, R}|} \cdot D \varphi_R = \beta \cdot D \varphi_R.
\]
Hence we get that 

\[
\displaystyle \frac{\partial }{\partial * } Df_L  \hspace{1cm} \mbox{and} \hspace{1cm} \frac{\partial }{\partial * } Df_R
\]
are bounded for $* \in \{ \alpha, \beta, b \}$. From this and from Lemma~\ref{derivatives of compositions with nonlinearity} the result follows.
\end{proof}

\noindent {\bf Proof of Lemma~\ref{lemma for Cmatrix}}.
Let us assume that $\sigma=-;$ the proof for $\sigma=+$ is similar. 
By the last four results we have that the partial derivatives of $\tilde \eta_L$ and $\tilde\eta_R$
with respect to $\alpha$ and $b$ are bounded.
It remains for us to show that $
\displaystyle\Big| \frac{\partial \tilde{\eta}_L}{\partial \beta}\Big|$ and $\displaystyle\Big|\frac{\partial \tilde{\eta}_R}{\partial \beta}\Big|$
are arbitrarily small at sufficiently deep renormalization levels.
Notice that we have 
$0_{k+1}^{+}= f_L^k(b-1)$ and 
$0_{k+2}^{-}= f_L^k \circ f_R (b)$, then
$
\displaystyle \frac{\partial 0_{k+1}^{+}}{\partial \beta} = 0 $
and
\[
\displaystyle \frac{\partial 0_{k+2}^{-}}{\partial \beta} = D f_L^k \circ f_R(b) \cdot \frac{\partial f_R}{\partial \beta }(b) = b \cdot D f_L^k \circ f_R(b)
\]
which goes to zero when the renormalization level goes to infinity. 
\qed

\subsection{The $D_{\underline{f}}$ matrix}

\begin{equation}
\label{partial derivatives D for sigma -}
D_{\underline{f}} = \left( \begin{array}{cc}
\displaystyle \frac{\partial \tilde{\eta}_L}{\partial \eta_L }     & \displaystyle \frac{\partial \tilde{\eta}_L}{\partial \eta_R } \\ 
&  \\
 \displaystyle \frac{\partial \tilde{\eta}_R}{\partial \eta_L }     & \displaystyle \frac{\partial \tilde{\eta}_R}{\partial \eta_R }  \end{array}           \right),
\end{equation}

\begin{lemma}
\label{lemma entries of D matrix}
Let $\underline{f} \in  \underline{\mathcal D}_0$. The maps

\begin{equation}
\begin{array}{ccl}
    {\mathcal C}^1([0, 1])^1 \ni (\eta_L, \eta_R) & \mapsto &  \tilde{\eta}_L \in {\mathcal C}^1([0, 1]) \\
    && \\
     {\mathcal C}^1([0, 1])^1 \ni (\eta_L, \eta_R) & \mapsto &  \tilde{\eta}_R \in {\mathcal C}^1([0, 1]) \\ 
\end{array}    
\end{equation}
are differentiable. Furthermore, for any $\varepsilon>0,$ and infinitely renormalizable
$\underline g\in\underline{\mathcal{D}}_0,$ we have that there exists $n_0\in\mathbb N,$ so that if $n\geq n_0,$ and $\underline f=\underline{\mathcal{R}}^n \underline g,$ we have that 
each $\displaystyle \Big|\frac{\partial\tilde\eta_i}{\partial\eta_j}\Big|<\varepsilon,$ for $i,j\in\{L,R\}.$
\end{lemma}

We will prove this lemma after some preparatory results.

\begin{lemma}
\label{lemma d Nf by d eta}
Let 

\begin{equation}
\begin{array}{ccccl}
G & : & \mbox{\em Diff}_+^1([0, 1]) & \rightarrow & \mathcal{C}^1([0, 1]) \\
  &   & \eta & \mapsto & G(\eta) \\  
\end{array}    
\end{equation}
be a $\mathcal{C}^1$ operator with bounded derivative. 
Let $\underline{f} \in \underline{\mathcal D}_0 $. The operators

\begin{equation}
\begin{array}{ccccl}
H_1, H_2 & : & \mbox{\em Diff}_+^3([0, 1]) & \rightarrow & \mathcal{C}^1([0, 1]) \\
&  & \eta_{\star} & \mapsto & \left\{ \begin{array}{l}
     H_1(\eta_{\star})=Nf_{L} \circ G(\eta_{\star})  \\
     H_2(\eta_{\star})=Nf_{R} \circ G(\eta_{\star}) 
\end{array}
\right. \\ 
\end{array}    
\end{equation}
where $\star \in \{ L, R \}$, are differentiable. 
\end{lemma}

\begin{proof}
Using the partial derivative operator $\partial $ we obtain 

\[
\displaystyle \frac{\partial }{\partial \eta_{\star}} \big[ H_{1}(\eta_{\star})\big] = \displaystyle \frac{\partial }{\partial \eta_{\star}} \big[ Nf_{L}\big] \circ G (\eta_{\star}) + D(Nf_{L}) \circ G(\eta_{\star}) \cdot \frac{\partial }{\partial \eta_{\star}} \big[ G(\eta_{\star}) \big]
\]
and
\[
\displaystyle \frac{\partial }{\partial \eta_{\star}} \big[ H_{2}(\eta_{\star})\big] = \displaystyle \frac{\partial }{\partial \eta_{\star}} \big[ Nf_{R}\big] \circ G (\eta_{\star}) + D(Nf_{R}) \circ G(\eta_{\star}) \cdot \frac{\partial }{\partial \eta_{\star}} \big[ G(\eta_{\star}) \big],
\]
with $\star \in \{ L, R \}$.
\end{proof}

\begin{lemma}
\label{lemma d D varphi bu d eta}
The operator $ F : \mbox{\em Diff}_{+}^{3}([0, 1])= {\mathcal C}^1([0, 1]) \rightarrow {\mathcal C}^1([0, 1])$

\begin{equation*}
\begin{array}{ccccl}
   F & : &  \eta & \mapsto &  F(\eta) = D \varphi_{\eta}(x)\\
\end{array}
\end{equation*}
is differentiable and its derivative is bounded.
\end{lemma}

\begin{proof}
Since the nonlinearity is a bijection, given a nonlinearity $\eta \in {\mathcal C}^1([0, 1])$ its corresponding diffeomorphism is given explicitly by

\[
\varphi_{\eta}(x) = \displaystyle \frac{\int_{0}^{x}e^{\int_{0}^{s} \eta (t) dt}ds}{\int_{0}^{1}e^{\int_{0}^{s} \eta (t) dt}ds},
\]
and the derivative of $\varphi_{\eta}(x)$ is

\[
D \varphi_{\eta}(x) = \displaystyle \frac{e^{\int_{0}^{x} \eta (t) dt}}{\int_{0}^{1}e^{\int_{0}^{s} \eta (t) dt}ds}.
\]
Thus, the derivative of $F$ can be calculated and is

\[
\displaystyle \frac{\partial }{\partial \eta} \big( D \varphi_{\eta}(x)\big) \Delta \eta = \frac{e^{\int_{0}^{x} \eta}}{\big( \int_{0}^{1}e^{\int_{0}^{s} \eta} ds \big)^2} \cdot 
\left[ \int_{0}^{1} e^{\int_{0}^{s} \eta } ds \cdot  \int_{0}^{x} \Delta \eta - \int_{0}^{1} [e^{\int_{0}^{s} \eta } \cdot \int_{0}^{s} \Delta \eta ]ds \right].
\]
From this expression it is possible to check and conclude that 
\[
\displaystyle \frac{\partial }{\partial \eta} \big( D \varphi_{\eta}(x)\big) \Delta \eta
\]
is bounded as we desire.
\end{proof}

\begin{cor}
\label{corollary d Df g by d eta}
Let 

\begin{equation}
\begin{array}{ccccl}
G & : & \mbox{\em Diff}_+^1([0, 1]) & \rightarrow & \mathcal{C}^1([0, 1]) \\
  &   & \eta & \mapsto & G(\eta) \\  
\end{array}    
\end{equation}
be a $\mathcal{C}^1$ operator with bounded derivative. 
Let $\underline{f} \in \underline{\mathcal D}_0 $. The operators

\begin{equation}
\begin{array}{ccccl}
H_{1}, H_2 & : & \mbox{\em Diff}_+^3([0, 1]) & \rightarrow & \mathcal{C}^1([0, 1]) \\
&  & \eta_{\star} & \mapsto & 
\left\{ 
\begin{array}{l}
H_{1}(\eta_{\star})=Df_{L} \circ G(\eta_{\star}) \\ 
H_{2}(\eta_{\star})=Df_{R} \circ G(\eta_{\star}) \\       
\end{array}
\right.
\end{array}    
\end{equation}
where $\star \in \{ L, R \}$, are differentiable and their derivatives are bounded.
\end{cor}

Now we can make the proof of Lemma~\ref{lemma entries of D matrix}.

{\bf Proof of Lemma~\ref{lemma entries of D matrix}}.

The proof will be done just for the case $\sigma_f = -$. The case $\sigma_f=+$ is analogous and we leave it to the reader. From (\ref{renormalized coordinates case -}) the partial derivatives of $\tilde{\eta}_L$ with respect to $\eta_L$ and $\eta_R$ are given by

\begin{equation}
\begin{array}{ccl}
    \displaystyle \frac{\partial \tilde{\eta}_L}{\partial \eta_L} & = & 
    \displaystyle \frac{\partial }{\partial \eta_L} \left( Z_{[0_{k+1}^+,0]} \eta_{\tilde{f}_L} \right) \\
    && \\
    & = & \displaystyle \frac{\partial }{\partial  0_{k+1}^+} \left( Z_{[0_{k+1}^+,0]} \eta_{\tilde{f}_L} \right) \cdot \displaystyle \frac{\partial }{\partial \eta_L} \left( 0_{k+1}^+ \right)
      + \displaystyle \frac{\partial }{\partial \eta_{\tilde{f}_L}} \left( Z_{[0_{k+1}^+,0]} \eta_{\tilde{f}_L} \right) \cdot \displaystyle \frac{\partial }{\partial \eta_L} \left( \eta_{\tilde{f}_L} \right)
\end{array}
\end{equation}
and

\begin{equation}
\begin{array}{ccl}
    \displaystyle \frac{\partial \tilde{\eta}_L}{\partial \eta_R} & = & 
    \displaystyle \frac{\partial }{\partial \eta_R} \left( Z_{[0_{k+1}^+,0]} \eta_{\tilde{f}_L} \right) \\
    && \\
    & = & \displaystyle \frac{\partial }{\partial  0_{k+1}^+} \left( Z_{[0_{k+1}^+,0]} \eta_{\tilde{f}_L} \right) \cdot \displaystyle \frac{\partial }{\partial \eta_R} \left( 0_{k+1}^+ \right)
      + \displaystyle \frac{\partial }{\partial \eta_{\tilde{f}_L}} \left( Z_{[0_{k+1}^+,0]} \eta_{\tilde{f}_L} \right) \cdot \displaystyle \frac{\partial }{\partial \eta_R} \left( \eta_{\tilde{f}_L} \right),
\end{array}
\end{equation}
respectively. From Lemma~\ref{lemma Zoom derivatives} we know that  
\[
\displaystyle \frac{\partial }{\partial  0_{k+1}^+} \left( Z_{[0_{k+1}^+,0]} \eta_{\tilde{f}_L} \right)
\]
is bounded and 

\[
\big| \big| \displaystyle \frac{\partial }{\partial \eta_{\tilde{f}_L}} \left( Z_{[0_{k+1}^+,0]} \eta_{\tilde{f}_L} \right) \big| \big| = |0_{k+1}^{+}| \rightarrow 0
\]
when the level of renormalization tends to infinity. Hence, $\big| \big| \displaystyle \frac{\partial }{\partial \eta_{\tilde{f}_L}} \left( Z_{[0_{k+1}^+,0]} \eta_{\tilde{f}_L} \right) \big| \big|$ is as small as we desire. From (\ref{d 0k+1 d eta_L}) (in the proof of Lemma~\ref{lemma entries of B_f matrix}) we have

\[
\displaystyle \frac{\partial }{\partial \eta_L} \big( 0_{k+1}^{+} \big) \asymp |T_{0,L}| \cdot \big( 1- 1_{I_{0,L}}^{-1}(f_L^{k-1}(b-1)) \big),
\]
 which is also as small as we desire. Since $0_{k+1}^{+}=f_L^k(b-1)$ does not depend on $\varphi_R$ we have 
 
 \[
 \displaystyle \frac{\partial }{\partial \eta_R} \big( 0_{k+1}^{+} \big)=0.
 \]
 Hence, in order to prove that 
 \[
  \displaystyle \frac{\partial \tilde{\eta}_L}{\partial \eta_L} \hspace{1cm} \mbox{and} \hspace{1cm}  \displaystyle \frac{\partial \tilde{\eta}_L}{\partial \eta_R}
 \]
are tiny we just need to prove that 

\[
\displaystyle |0_{k+1}^{+}| \cdot \big| \frac{\partial }{\partial \eta_L} \left( \eta_{\tilde{f}_L} \right)\big| \hspace{1cm} \mbox{and} \hspace{1cm} \displaystyle |0_{k+1}^{+}| \cdot \big| \frac{\partial }{\partial \eta_R} \left( \eta_{\tilde{f}_L} \right) \big|
\]
are tiny. Since $\eta_{\tilde{f}_L}=N(\tilde{f}_L)=N(f_L^k \circ f_R \circ f_L) $ from (\ref{derivative of eta_{tilde{f}_L}}) we obtain

\begin{equation}
\label{d eta f_L tilde by d eta_L}
\begin{array}{ccl}
\displaystyle \frac{\partial }{\partial \eta_L} \big( \eta_{\tilde{f}_L} \big)     & = & \displaystyle \sum_{i=1}^{k}  \left\{ \frac{\partial }{\partial \eta_L} \big[ N f_L (f_L^{k-i} \circ f_R \circ f_L) \big] \cdot Df_L^{k-i} \circ f_R \circ f_L  \cdot D(f_R \circ f_L)  \right. \\
   & & \\ 
     &  & \hspace{0.7cm} \displaystyle + N f_L (f_L^{k-i} \circ f_R \circ f_L) \cdot \frac{\partial }{\partial \eta_L} \big[  Df_L^{k-i} \circ f_R \circ f_L \big] \cdot D(f_R \circ f_L) \\
     & &  \\
     & & \left. \hspace{0.7cm} \displaystyle + N f_L (f_L^{k-i} \circ f_R \circ f_L) \cdot Df_L^{k-i} \circ f_R \circ f_L \cdot \frac{\partial }{\partial \eta_L} \big[ D(f_R \circ f_L) \big] \right\} \\
     & & \\
     & & + \displaystyle \frac{\partial }{\partial \eta_L} \big[ Nf_R \circ f_L \big] \cdot Df_L + Nf_R \circ f_L \cdot \frac{\partial }{\partial \eta_L} \big[ Df_L \big] + \frac{\partial }{\partial \eta_L} \big[ Nf_L \big]. \\ 
\end{array}    
\end{equation}
Since our gap mappings $f=(f_L, f_R, b)$ have bounded Schwarzian derivative $Sf$ and bounded nonlinearity $Nf$, by the formula for the Schwarzian derivative of $f$

\[
\displaystyle Sf = D(Nf)- \frac{1}{2}(Nf)^2 
\]
we obtain that $D(Nf_{L})$ and $D(Nf_{R})$ are bounded.
As
\[
Nf_{L} = \displaystyle \frac{1}{|I_{0, {L}}|} \cdot N \varphi_{L} \circ 1_{I_{0, {L}}} ^{-1} \hspace{1cm} \mbox{and} \hspace{1cm}
Nf_{R} = \displaystyle \frac{1}{|I_{0, {R}}|} \cdot N \varphi_{R} \circ 1_{I_{0, {R}}} ^{-1} 
\]
we have

\[
\displaystyle \frac{\partial }{\partial \eta_{\star}} \big[ Nf_{L}\big] = \frac{1}{|I_{0, {L}}|} \cdot \frac{\partial }{\partial \eta_{\star}} \big[ N \varphi_{L} \big] \circ 1_{I_{0, {L}}} ^{-1} = \frac{1}{|I_{0, {L}}|} \cdot \frac{\partial }{\partial \eta_{\star}} \big[ \eta_{\varphi_{L} } \big] \circ 1_{I_{0, {L}}} ^{-1},
\]
and

\[
\displaystyle \frac{\partial }{\partial \eta_{\star}} \big[ Nf_{R}\big] = \frac{1}{|I_{0, {R}}|} \cdot \frac{\partial }{\partial \eta_{\star}} \big[ N \varphi_{R} \big] \circ 1_{I_{0, {R}}} ^{-1} = \frac{1}{|I_{0, {R}}|} \cdot \frac{\partial }{\partial \eta_{\star}} \big[ \eta_{\varphi_{R} } \big] \circ 1_{I_{0, {R}}} ^{-1},
\]
where $\star \in \{ L, R \}$ and at this point we are calling $\eta_{\star} = \eta_{\varphi_{\star}}$ for sake of simplicity. As

\[
\displaystyle Df_L = \frac{|T_{0,L}|}{|I_{0,L}|} \cdot D \varphi_L
\hspace{1cm} \mbox{and} \hspace{1cm}
\displaystyle Df_R = \frac{|T_{0,R}|}{|I_{0,R}|} \cdot D \varphi_L,
\]
we obtain that the product
\[
\displaystyle \frac{\partial }{\partial \eta_L} \big[ N f_L (f_L^{k-i} \circ f_R \circ f_L) \big]  \cdot D(f_R \circ f_L)
\]
is bounded. From Corollary~\ref{corollary d Df g by d eta} we obtain that all the terms

\[
\displaystyle \frac{\partial }{\partial \eta_L} \big[  Df_L^{k-i} \circ f_R \circ f_L \big], \;\; \frac{\partial }{\partial \eta_L} \big[ D(f_R \circ f_L) \big] \;\; \mbox{and} \;\;\frac{\partial }{\partial \eta_L} \big[ Df_L \big]
\]
are also bounded. From Lemma~\ref{lemma derivatives of f_L and f_R} we obtain that 
\[
\displaystyle \frac{\partial }{\partial \eta_L} \big( f_L \big)
\]
is bounded. Furthermore, we know that
\[
\displaystyle |0_{k+1}^{+}| \cdot \frac{\partial }{\partial \eta_L} \big[ Nf_L \big] \longrightarrow 0
\]
when the level of renormalization tends to infinity. Hence, using Lemma~\ref{lemma derivatives of f_L and f_R}, Lemma~\ref{lemma d Nf by d eta}, Lemma~\ref{lemma d D varphi bu d eta} and Corollary~\ref{corollary d Df g by d eta} we conclude that 

\[
\displaystyle |0_{k+1}^{+}| \cdot \big| \frac{\partial }{\partial \eta_L} \left( \eta_{\tilde{f}_L} \right) \big|
\]
is tiny. Analogously, we obtain that

\[
\displaystyle |0_{k+1}^{+}| \cdot \big| \frac{\partial }{\partial \eta_L} \left( \eta_{\tilde{f}_R} \right) \big|
\]
is also tiny, which completes the proof of Lemma~\ref{lemma entries of D matrix}, as desired.


\section{Manifold structure of the conjugacy classes}
\subsection{Expanding and contracting directions of $D\underline{\mathcal{R}}_{\underline f}$}

Let $\underline f_n$ be the $n$-th renormalization of an infinitely renormalizable dissipative gap mapping in the decomposition space. In this section, we will assume that $\sigma_{f_n}=-$. The case when $\sigma_{f_n}=+$ is similar. For any $\varepsilon>0$, there exists $n_0\in\mathbb N$ so that for $n\geq n_0$ we have that

$$D\underline{\mathcal{R}}_{\underline  f_n}\asymp\left[\begin{array}{ccccc}
\varepsilon & \varepsilon & \varepsilon & \varepsilon &\varepsilon\\
\varepsilon & \varepsilon & \varepsilon & \varepsilon &\varepsilon\\
K_1 & K_2 & \frac{\partial \tilde b}{\partial b} & 
\frac{\partial \tilde b}{\partial \eta_L} & 0 \\ 
C_1 & \varepsilon & \frac{\partial\tilde \eta_L}{\partial b}
&\varepsilon &\varepsilon \\
C_2 & C_3 & \frac{\partial\tilde \eta_R}{\partial b} &\varepsilon &\varepsilon \\
\end{array}\right],$$
where $K_i$ are large for $i\in\{1,2\}$ and $C_j$
are bounded for $j\in\{1,2,3\}.$
We highlight the partial derivatives that will be important in
the following calculations.
Let 
$$
K_3=\partial \tilde b/\partial b, \quad 
    K_4={\partial \tilde b}/{\partial \eta_L}$$
  $$  M_1=\partial\tilde \eta_L/\partial b,\quad \mbox{ and } \quad
    M_2=\partial\tilde \eta_R/\partial b.
$$

\begin{prop}
For any $\delta>0$, there exists $n_0\in\mathbb N,$ so that
for all $n\geq n_0$, we have the following: 
\begin{itemize}
\item $T_{\underline{\mathcal{R}}_{\underline{\mathcal{R}}^n\underline f}} \underline{\mathcal{D}}=E^u\oplus E^s,$ and the subspace $E^u$ is one dimensional. 

\item For any vector $v\in E^u$, we have that 
$\|D\underline{\mathcal{R}}_{\underline{\mathcal{R}}^n\underline f}v\|\geq \lambda_1\|v\|$, where $|\lambda_1|>1/\delta$.
\item For any
 $v\in E^s$, we have that 
$\|D\underline{\mathcal{R}}_{\underline{\mathcal{R}}^n\underline f}v\|\leq \lambda\|v\|$, where $|\lambda|<\delta$.
\end{itemize}

\end{prop}

\begin{proof}
By taking $n$ large, we can assume that $\varepsilon$ is arbitrarily small.
To see that for $\varepsilon$ sufficiently small 
the tangent space admits a hyperbolic splitting, it is enough to check that this holds for the matrix:
$$\left[\begin{array}{ccccc}
0 & 0 & 0 & 0& 0\\
0& 0 & 0& 0& 0\\
K_1 & K_2 & K_3 & K_4 & 0\\ 
C_1 & 0 & M_1 &0&0\\
C_2 & C_3 & M_2 &0&0 \\
\end{array}\right].$$
Calculating
$$\mathrm{det}\left[\begin{array}{ccccc}
\lambda & 0 & 0 & 0& 0\\
0& \lambda & 0& 0& 0\\
-K_1 & -K_2 & \lambda-K_3 & -K_4 & 0\\ 
C_1 & 0 & -M_1 &\lambda&0\\
C_2 & C_3 & -M_2 &0&\lambda \\
\end{array}\right]=\lambda^2\mathrm{det}
\left[\begin{array}{ccc}
 \lambda-K_3& -K_4 & 0\\ 
 -M_1 &\lambda&0\\
-M_2 &0&\lambda \\
\end{array}\right]$$
$$= \lambda^2\big(     
(\lambda-K_3)\lambda^2+K_4(-M_1\lambda))
\big)$$
$$=\lambda^3 \big(     
(\lambda-K_3)\lambda+K_4(-M_1)
\big)$$
has zero as a root with multiplicity three, and the remaining roots are
the zeros of the quadratic polynomial
$\lambda^2- K_3\lambda -K_4 M_1,$
which are given by
$$\frac{K_3\pm\sqrt{K_3^2+4K_4 M_1}}{2}.$$
We immediately see that 
$\frac{K_3+\sqrt{K_3^2+4K_4 M_1}}{2}$ is much bigger than one, when $K_3=\partial \tilde b/\partial b$ is large.

Now, we show that
$$\Big|\frac{K_3-\sqrt{K_3^2+4K_4 M_1}}{2}\Big|
=\frac{\sqrt{K_3^2+4K_4 M_1}-K_3}{2}
$$
is small.

We have that
$$\frac{\sqrt{K_3^2+4K_4 M_1}-K_3}{2}=
\frac{K_3}{2}\Big(\sqrt{1+4\frac{K_4 M_1}{K_3^2}}-1\Big).$$

By equations~(\ref{eqn:57}) and (\ref{eqn:72}), we have
that 
$$\Big|\frac{K_4}{K_3}\Big|\leq \frac{b/|I'|+C'}{1/3|I|'}\leq C b,$$
where $C,C'$ are bounded.
For deep renormalizations we have that $b$ is arbitrarily close to zero,
for otherwise $0$ is contained in the gap $(f_R(b), f_L(b-1)),$
which is close to $(b-1,b)$ at deep renormalization levels.

Thus we have that
$$\frac{K_3}{2}\Big(\sqrt{1+4\frac{K_4 M_1}{K_3^2}}-1\Big)
\leq \frac{K_3}{2}\Big(\sqrt{1+4Cb\frac{M_1 +M_2}{K_3}}-1\Big).$$
For large $K_3$, by L'Hopital's rule, we have that this is approximately,
$$Cb \frac{M_1 + M_2}{\sqrt{1+4Cb \frac{M_1+M_2}{K_3}}}.$$
Finally by Corollary~\ref{cor:bdd N}, we have that
$M_1+M_2$ is bounded. Hence for deep renormalizations, 
$$\Big|\frac{K_3-\sqrt{K_3^2+4K_4 M_1}}{2}\Big|$$
is close to zero.
\end{proof}

\subsection{Cone Field}
Recall our expression of
$$D\underline{\mathcal{R}}_{\underline f_n}\quad\mbox{as}\quad
\left[\begin{array}{cc} A_{\underline f_n} & B_{\underline f_n}\\ C_{\underline f_n} & D_{\underline f_n}
\end{array} \right].$$
We will omit the subscripts when it will not cause confusion.

For $r\in(0,1)$, we define the cone
$$C_r=\{(\Delta\alpha,\Delta\beta,\Delta b)\in(0,1)^3:\Delta\alpha+\Delta\beta\leq r\Delta b \}.$$
Note that we regard cones as being contained in the tangent space of the
decomposition space.

\begin{lemma}\label{lem:cone1}
For any $\lambda_0>1,$ and every $r\in(0,1),$
there exists $n_0$, so that for all $n\geq n_0$
the cone $C_r$ is invariant and expansive; that is,
\begin{itemize}
\item $A_{\underline f_n}(C_r)\subset C_{r/3},$ and
\item if $v\in C_r$, then $|A_{\underline f_n}v|>\lambda_0|v|.$
\end{itemize}
\end{lemma}
\begin{proof}
For all $n$ sufficiently large we have that 
$A_{\underline f_n}$ is of the order
$$\left[\begin{array}{ccc}
\varepsilon & \varepsilon & \varepsilon \\
\varepsilon & \varepsilon & \varepsilon \\
K_1 & K_2 & \frac{\partial \tilde b}{\partial b}
\end{array}\right]
$$

Let $\Delta v=(\Delta\alpha,\Delta \beta,\Delta b)\in C_r,$
and $\Delta\tilde v=(\Delta\tilde\alpha,\Delta\tilde\beta,\Delta\tilde b)=A_{\underline f_n}\Delta v$.

To see that the cone is invariant, we estimate
$$\frac{|(\Delta \tilde\alpha,\Delta\tilde\beta)|}{|\Delta \tilde b|}\leq \frac{2\varepsilon(|\Delta\alpha +\Delta\beta+\Delta b| )}{K_3|\Delta b|}\leq r/3,$$
provided that
$$\frac{1+r}{r}\leq \frac{K_3}{6\varepsilon}.$$

To see that the cone is expansive, we estimate
$$\frac{|\Delta \tilde v|}{|\Delta v|}
\geq\frac{|\Delta\tilde b|}{|\Delta\alpha +\Delta \beta|+|\Delta b|}
\geq \frac{K_3|\Delta b|}{(1+r)|\Delta b|}=\frac{K_3}{1+r}
\geq \lambda_0$$ when $K_3$ is sufficiently large.
\end{proof}

\begin{lemma}\label{lem:cone2}
For all $0<r<1/2$ and every $\lambda>0$, there exists $\delta>0$
such that
$$C_{r,\delta}=\{\underline f\in\underline{\mathcal{D}}:|\Delta\eta_{L}|,|\Delta\eta_R|\leq \delta\Delta b,
\Delta\alpha+\Delta\beta<r\Delta b\}$$
is a cone field in the decomposition space. Moreover, if $\underline f\in\underline{\mathcal{D}}$ is an infinitely renormalizable dissipative gap
 mapping, 
then for all $n$ sufficiently big
\begin{itemize}
\item $D\underline{\mathcal{R}}_{\underline f_n}(C_{r,\delta})\subset C_{r/2,\delta/2},$ and
\item if $v\in C_{r,\delta}$, then $|D\underline{\mathcal{R}}_{\underline f}v|>\lambda|v|.$
\end{itemize}
\end{lemma}
\begin{proof}
Set $\Delta v=(\Delta \alpha, \Delta \beta, \Delta b, 
\Delta \eta_L,\Delta \eta_R),$
$\Delta X=(\Delta \alpha, \Delta \beta, \Delta b),$ and
$\Delta \Phi=(\Delta \eta_L,\Delta \eta_R).$ As before, we
mark the corresponding objects under renormalization with a
tilde.
Then we have that 
$$D\underline{\mathcal{R}}_{\underline f_n}(\Delta v)=
\left[\begin{array}{cc}
A & B\\
C & D
\end{array}\right]
\left[\begin{array}{c} \Delta X \\ \Delta\Phi\end{array}\right]=
\left[\begin{array}{c}
A\Delta X + B\Delta\Phi\\
C\Delta X+ D\Delta\Phi\end{array}\right].$$

We let $(\Delta\hat\alpha, \Delta\hat\beta, \Delta \hat b)=A\Delta X.$

First, we show that $|\Delta \tilde b|$ is much bigger than
$\Delta b.$ By Lemma~\ref{lem:cone1}, we have that
$$|\Delta\hat\alpha|+ |\Delta\hat\beta|+|\Delta \hat b|
\geq\lambda_0(|\Delta\alpha|+ |\Delta\beta|+|\Delta b|)
\quad\mbox{and}\quad |\Delta\hat\alpha|+|\Delta\hat\beta|\leq\frac{r}{3}|\Delta\hat b|,$$
where we can take $\lambda_0>0$ arbitrarily large.
Thus we have that $(1+r/3)|\Delta \hat b|\geq \lambda_0|\Delta b|,$
and so, since $r\in(0,1),$ 
$$|\Delta\hat b|\geq\frac{3}{4}\lambda_0|\Delta b|.$$
To see that  $|\Delta \tilde b|$ is much bigger than
$\Delta b$ observe that $|\Delta \tilde b-\Delta \hat b|\leq
\varepsilon(\Delta \eta_L+\Delta \eta_R)<2\varepsilon\delta |\Delta b|.$
So 
\begin{eqnarray*}
|\Delta \tilde b|=|\Delta \tilde b-\Delta \hat b+\Delta \hat b|
\geq |\Delta\hat b|-|\Delta\tilde b-\Delta \hat b|\\
\geq \frac{3}{4}\lambda_0|\Delta b|-2\varepsilon\delta |\Delta b|
\geq \frac{\lambda_0}{2}|\Delta b|,
\end{eqnarray*}
when $\lambda_0$ is large enough.

Now, we prove that the cone is invariant. First of all, we have
$$|\Delta\tilde\alpha+\Delta\tilde\beta|
\leq |\Delta\hat\alpha+\Delta\hat\beta|+|B\Delta\Phi|
\leq \frac{r}{3}|\Delta \hat b|+2\varepsilon\delta |\Delta b|$$
$$
\leq \frac{r}{3}|\Delta \tilde b|+4\varepsilon\delta |\Delta b|
\leq \frac{r}{3}|\Delta \tilde b|+\frac{8\varepsilon\delta}{\lambda_0} |\Delta \tilde b|
\leq \frac{r}{2}|\Delta\tilde b|,$$
for $\lambda_0$ large enough.
Second, we have that 
$$\Delta \tilde \Phi = C\Delta X +D\Delta\Phi,$$ where the entries of
$C$ and $D$ are bounded, say by $K>0$,
so that
\begin{eqnarray*}
|\Delta\tilde \eta_L|+|\Delta\tilde \eta_R|
&\leq& K(|\Delta\alpha| +|\Delta \beta|+|\Delta b| +
|\Delta \eta_L|+|\Delta\eta_R|)\\
&\leq& K(1+r+\delta)|\Delta b|\\
&\leq& 2\frac{K(1+r+\delta)}{\lambda_0}|\Delta \tilde b|\\
&\leq &\frac{\delta}{2}|\Delta\tilde b|,
\end{eqnarray*}
for $\lambda_0$ sufficiently large.

Now let us show that the cone is expansive.
\begin{eqnarray*}
|D\underline{\mathcal{R}}_{\underline f_n}\Delta v|&\geq& |A\Delta X+D\Delta\Phi|\geq |A\Delta X|-|B\Delta\Phi|\\
&\geq& \lambda_0|\Delta X|-\varepsilon\delta|\Delta b|\\
&\geq&\lambda_0(|\Delta b|-|\Delta\alpha+\Delta\beta|)
-\varepsilon\delta|\Delta b|\\
&\geq&\lambda_0(|\Delta b|-r|\Delta b|)-\varepsilon\delta|\Delta b|\\
&\geq& (\lambda_0(1-1/2)-\varepsilon\delta)|\Delta b|\\
&\geq & \frac{\lambda_0}{3}|\Delta b|,
\end{eqnarray*}
for $\delta$ small enough.
We also have that 
$$|\Delta v|\leq |\Delta\alpha| +|\Delta \beta|+|\Delta b| +|\Delta \eta_L|+|\Delta\eta_R|$$
$$\leq (r+1+\delta)|\Delta b|.$$
Hence
$$\frac{\Delta \tilde v}{\Delta v}\geq\frac{\lambda_0/3}{r+1+\delta},$$
which we can take as large as we like.
\end{proof}

\begin{lemma}
\label{Technical Lemma}
Let $\underline{f} \in \underline{\mathcal{D}}$ be a renormalizable dissipative gap mapping. If $\Delta \tilde{v} = D \underline{\mathcal{R}}_{\underline{f}} \big( \Delta v \big) \notin C_{r, \delta}$, then there exists a constant $K>0$ such that

\begin{enumerate}
    \item[(i)] $| \Delta b | \leq K \cdot |I '| \cdot ||\Delta v ||$,
    \\
    \item[(ii)] $|| \Delta \tilde{v} || \leq K ||\Delta v||$,
\end{enumerate}
where $I'$ is the domain of the renormalization $\underline{\mathcal{R}}_{\underline{f}}$ before rescaling which is defined on page 11.
\end{lemma}

\begin{proof}

For convenience in this proof we express $\underline{f}$ in new coordinates, $\underline{f} = (b,x)$, where $x=(\alpha, \beta, \eta_L, \eta_R )$. We use the same notation for a vector $\Delta v  =(\Delta b, \Delta x)$, where $\Delta x = (\Delta \alpha, \Delta \beta, \Delta \eta_L, \Delta \eta_R)$.
Since 
$\Delta \tilde{v} = D \underline{\mathcal{R}}_{\underline{f}} \big( \Delta v \big)$ it is not difficult to check that
\[
\displaystyle \Delta \tilde{b} = K_1 \cdot \Delta \alpha + K_2 \cdot \Delta \beta + \frac{\partial \tilde{b}}{\partial b} \cdot \Delta b + \frac{\partial \tilde{b}}{\partial \eta_L} \cdot \eta_L
+0 \cdot \Delta \eta_R. 
\]
Using Lemmas~\ref{lemma partial derivatives of abb tilde respec to abb} and \ref{lemma entries of B_f matrix} we get

\begin{equation}
\label{Delta b tilde to Delta b}
\displaystyle \frac{|\Delta \tilde{b}|}{|\Delta b|} \asymp \frac{1}{|I'|}.
\end{equation}
From the hypothesis $\Delta \tilde{v} = \big( \Delta \tilde{b}, \Delta \tilde{x} \big) = D \underline{\mathcal{R}}_{\underline{f}} \big( \Delta v \big) \notin C_{r, \delta}$ we have

\begin{equation}
\label{Delta b tilde} 
|\Delta \tilde{b}| \leq C \cdot ||\Delta \tilde{x}||
\end{equation}
for some constant $C>0$. This inequality together with (\ref{Delta b tilde to Delta b}) imply in

\[
|\Delta b| \asymp C \cdot |I'| \cdot || \Delta v||, 
\]
which proves statement (i). For statement (ii) we just observe that except for two entries on third line of matrix 

\[
D\underline{\mathcal{R}}_{\underline f_n} =
\left[\begin{array}{cc} A_{\underline f_n} & B_{\underline f_n}\\ C_{\underline f_n} & D_{\underline f_n}
\end{array} \right]
\]
all the others entries are bounded. Then we obtain

\begin{equation}
\label{x tilde}
|| \Delta \tilde{x} || = O \big( || \Delta v|| \big).
\end{equation}
Since 
\[
|| \Delta \tilde{v} || = |\Delta \tilde{b}| + || \Delta \tilde{x}, || 
\]
from (\ref{Delta b tilde}) we obtain

\[
|| \Delta \tilde{v} || \leq C \cdot || \Delta \tilde{x} || + ||\Delta \tilde{x} ||
\]
and from (\ref{x tilde}) we are done.
\end{proof}

\subsection{Conjugacy classes are $\mathcal C^1$ manifolds}

Let $\underline f\in\underline{\mathcal{D}}$ be an infinitely renormalizable gap mapping, regarded as an
element of the decomposition space. Let $\underline{\mathcal{T}}_{\underline f}\subset \underline{\mathcal{D}}$, be the 
topological conjugacy class of $\underline f$ in
$\underline{\mathcal{D}}$. 

Observe that for $M>0$ sufficiently large and 
$\varepsilon>0$ sufficiently small
$$B_0=\{(\alpha, \beta,\eta_L,\eta_R):|\eta_L|,|\eta_R|<M;\alpha,
\beta<\varepsilon\}$$
is an absorbing set for the renormalization operator acting on the 
decomposition space; that is, for every infinitely renormalizable
$\underline f\in\underline{\mathcal{D}},$
there exists $M>0$ 
with the property that for any
$\varepsilon>0,$ there exists
$n_0\in\mathbb N,$ 
so that for any 
for $n\geq n_0$, $\underline{\mathcal{R}}^n\underline f\in B_0.$

To conclude the proof of Theorem~\ref{thm:A}, we make use of the 
graph transform.  
We refer the reader to Section~2
of \cite{MP}, for the proofs of some of the results in this section.
Let $$X_0=\{w\in C(B,[0,1]): \mbox{for all } p,q\in
\mbox{graph}(w), q-p\notin C_{r,\delta}\}.$$

A $\mathcal C^1$ curve $\gamma:[0,1]\to \mathcal D$ is called 
{\em almost horizontal} if the tangent vector
$ T_{\gamma(\xi)}\gamma(\xi)\in C_{r,\delta}$, for all $\xi \in (0,1)$ with
$\gamma(0)=(\alpha,\beta,0,\eta_L,\eta_R)$, and
$\gamma(1)=(\alpha,\beta,1,\eta_L,\eta_R)$.
Notice that for any almost horizontal curve $\gamma$, 
and $w\in X_0$, there is a unique point
$w^\gamma=\gamma\cap\mathrm{graph}(w)$.
For any $p,q\in\gamma$, we set 
$\ell_{\gamma}(p,q)$ to be the length of the shortest 
curve in $\gamma$ connecting $p$ and $q$.

For $w_1,w_2\in X_0$, let 
$$d_0(w_1,w_2)=\sup_{\gamma}\ell_{\gamma}(w_1^\gamma,w_2^\gamma).$$
It is easy to see that $d_0$ is a complete metric on $X_0$.
Let $w\in X_0$, $\psi\in B_0$ and let
$\gamma_{\psi}$ be the horizontal line at $\psi$.
Then
there exists a subcurve of $\gamma_\psi$ corresponding to a renormalization window that
is mapped to an almost horizontal curve $\tilde \gamma$ under
renormalization.

We define the graph transform by
$$Tw(\psi)=\underline{\mathcal{R}}^{-1}((\underline{\mathcal{R}}w)^{\tilde\gamma}).$$
By \cite{AlvesColli2}, we have that
if $\underline f_b=(\alpha, \beta, b,\eta_L,\eta_R)$ and 
$\underline f_{b'}=(\alpha, \beta, b',\eta_L,\eta_R),$ 
are two, $n$ times renormalizable dissipative gap mappings
with the same combinatorics, then for every $\xi\in[b,b']$, we have that
$(\alpha, \beta, \xi,\eta_L,\eta_R)$ is $n$-times renormalizable with the same
combinatorics.
It follows from the invariance of the cone field that
$Tw\in X_0$, and by Lemma~\ref{lem:cone2}
we have that $T$ is a contraction.
From these considerations, we have that 
$T$ has a fixed point $w^*$ and that 
the graph of $w^*$ is contained in 
$\{(\alpha, \beta,b,\eta_L,\eta_R)\in\mathcal D:(\alpha,\beta,\eta_L,\eta_R)\in B_0\}$.

\begin{prop}
\label{prop:deep manifold}
We have that $\underline{\mathcal{T}}_{\underline f}\cap B_0$ 
is a $\mathcal C^1$ manifold.
\end{prop}
To prove this proposition, we use the graph transform acting to plane fields to show that $\underline{\mathcal{T}}_{\underline f}\cap B_0$ has a continuous field of tangent planes.

A {\em plane} is a codimension $1$ subspace of $\mathbb{R} \times B_0$ which is the graph of a functional $b^* \in \mathrm{Dual}(B_0)$. By identifying the plane with the corresponding functional $b^*$ we have that $\mathrm{Dual}(B_0)$ is the space of planes and carries a corresponding complete distance $d^*_{B_0}$.

Let us fix a constant $\chi>0$ to be chosen later.
\begin{definition}
\label{Def plane admissible}
Let $p=\underline{f} \in \mathrm{graph}(w^*)$. A plane $V_p$ is admissible for $p$ if it has the following properties:

\begin{enumerate}
    \item[(1)] if $(\Delta \alpha, \Delta \beta, \Delta b, \Delta \eta_L, \Delta \eta_R) \in V_p$ then $|\Delta b| \leq \chi b ||(\Delta \alpha, \Delta \beta, \Delta \eta_L, \Delta \eta_R )||$,
    
    \item[(2)] $V_p$ depends continuously on $p$ with respect to $d^*_{B_0}$.
\end{enumerate}
The set of admissible planes for $p$ is denoted by $\mathrm{Dual}_p(B_0)$.
\end{definition}

We let $X_1$ denote the space of all admissible plane fields. For clarity of exposition, 
we will express $\underline{f}$ in new coordinates: $\underline{f} = (b,x)$, where $x=(\alpha, \beta, \eta_L, \eta_R )$. We use the same notation for a vector $\Delta v  =(\Delta b, \Delta x)$, where $\Delta x = (\Delta \alpha, \Delta \beta, \Delta \eta_L, \Delta \eta_R)$, and although $V_{\underline{f}}^*$ is a subspace of $\mathbb{R} \times B_0$, for the next result we abuse notation and denoting the set $\{p+v | v \in  V_{\underline{f}}^* \}$ also by $V_{\underline{f}}^*$. 

Let $p=(b,x)\in w^*$ and define a distance on
$\mathrm{Dual}_p(B_0)$ as follows. For any two planes, 
$V_p,V'_p\in\mathrm{Dual}_p(B_0),$ let 
$\mathcal S$ denote the set of all straight lines
$\gamma$ with direction in 
$C_{r,\delta}.$ Provided that $\varepsilon$ is small enough, 
$\gamma$ intersects $V_p$ at exactly one point, and 
likewise for $V'_p$.
Let $\Delta q_{\gamma}=(\Delta b_{\gamma},\Delta x_{\gamma})=\gamma\cap V_p$, and $\Delta q'_{\gamma}=(\Delta b'_{\gamma},\Delta x'_{\gamma})=\gamma\cap V_p'.$
We define
$$d_{1,p}(V_p,V'_p)=
\sup_{\gamma\in\mathcal S}\frac{|\Delta b_{\gamma}-\Delta b'_{\gamma}|}{\min\{|\Delta q_{\gamma}|,|\Delta q'_{\gamma}|\}}.$$
When it will not cause confusion we will omit $\gamma$ from the notation. 
It is not hard to see that $d_{1,p}$ is a complete metric.
For $V,V'\in X_1,$ we define
$$d_1(V,V')=\sup_{p\in w^*}d_{1,p}(V_p,V'_p).$$
On an absorbing set for renormalization operator, we have that 
$d_1$ is metric and $(X_1,d_1)$ is a complete metric space.
This follows just as in
\cite[Lemmas 2.29 and 2.30]{MP}.

We define the graph transform $Q:X_1\to X_1$ by
$$Q V_{\underline f}=D\underline{\mathcal{R}}_{\underline{\mathcal{R}}\underline f}^{-1}(V_{\underline{\mathcal{R}}\underline f}).$$

\begin{lemma}
Admissible plane fields are invariant under $Q$. Moreover, $Q$ is contraction on the space $(X_1,d_1).$
\end{lemma}
\begin{proof}
Let us set $p=\underline f.$
To show invariance, assume that $V_p$ is an admissible plane field, and take
$(\Delta b,\Delta x)\in QV_p.$ Set
$(\Delta \tilde b,\Delta \tilde x)=
D\underline{\mathcal{R}}_p(\Delta b,\Delta x)\in V_{\mathcal R(p)}.$
By Lemma~\ref{Technical Lemma},
we have that
$\|\Delta b\|\leq K|I'|\|\Delta v\|,$ but now, since $V_{\underline{\mathcal{R}}(p)}$ is
an admissible plane field, we have that
 $K|I'|\|\Delta v\|\leq K_1|I'|\|\Delta x\|,$ where $K_1=K_1(K,r,\delta).$ 
 Furthermore, if $QV_p$ is not continuous in $p$, then there exists a sequence $p_n \rightarrow p$ such that $QV_{p_n}$ does not converge to $QV_p$. But now, since $QV_{p_n}$ and $QV_p$ are all codimension-one subspaces there exists $\Delta v \in QV_{p}$ such that $\Delta v$ is transverse to $QV_{p_n}$ for all $n$ sufficiently large. Since $V_{\underline{\mathcal{R}}(p)}$ is admissible, $D \underline{\mathcal R} \Delta v \in V_{\underline{\mathcal{R}}(p)}$. On the other hand, we can express $\Delta v=\Delta z'\oplus\Delta z$ with 
 $\Delta z\in C_{r,\delta}.$ 
 By the invariance of the cone field,
 we have that $\Delta\tilde v=\Delta y'\oplus\Delta y$ with $\Delta y\in C_{r,\delta}$. But now, $\Delta\tilde v$
 is transverse to $V_{\underline{\mathcal{R}}(p_n)}$
for all $n$ sufficiently big, which contradicts the admissibility of $V_p.$ 
 Hence we have that $QV_p$ depends continuously on $p$.

To see that $Q$ is a contraction, take two admissible plane fields $V,V'$, and line $\gamma\in\mathcal S.$
Define $\Delta q=(\Delta b,\Delta x)\in V$ and
$\Delta q'=(\Delta b',\Delta x')\in V$ be as in the definition of $d_{1,p}.$ 
Let $\Delta \tilde q=(\Delta \tilde b,\Delta \tilde x)=D\underline{\mathcal{R}}_{p}\Delta q$, and likewise for the objects marked with a prime. 
Observe that by 
Lemma~\ref{Technical Lemma}, we have that
$\|\Delta q\|\geq \frac{1}{C_1}\tilde\|\Delta \tilde q\|,$
and that
$|\Delta b|\leq C_2|I'||\Delta\tilde b|$.
So
$$\frac{|\Delta b-\Delta b'|}{\min\{\|\Delta q\|,\|\Delta q'\|\}}\leq C|I'|\frac{|\Delta\tilde b-\Delta\tilde b'|}{\min\{\|\Delta\tilde v\|,\|\Delta\tilde v\|\}}\leq \frac{1}{2}d_{1,\underline{\mathcal{R}}(p)}(V_{\underline{\mathcal{R}}({p)}},
V_{\underline{\mathcal{R}}(p)}').$$
Thus,
$$d_{1}(QV,QV')\leq\frac{1}{2}
d_1(V,V').$$
\end{proof}

Thus we have that there is an admissible plane field $V^*(\underline f),$ which is invariant plane field under $Q$.

\medskip

Now we conclude the proof of the proposition. 
We will show that for each $\underline f\in\mathrm{graph}(w^*),$
$V^*(\underline f)=T_{\underline f}(\mathrm{graph}(w^*)).$ 

Let $p \in \mathrm{graph}(w^*)$ and take an almost horizontal curve $\gamma$ close enough to $p$ such that $\gamma \cap \mathrm{graph}(w^*) = q =\{ p+ \Delta q = p+ (\Delta \alpha, \Delta \beta, \Delta b, \Delta \eta_L, \Delta \eta_R) \}$ and $\gamma \cap V_{\underline{f}}^* = q'= \{ p+ \Delta q' = p+ (\Delta \alpha', \Delta \beta', \Delta b', \Delta \eta_L', \Delta \eta_R') \}$. We define

\[
A = \sup_{p} \limsup_{\gamma \rightarrow p} \frac{|\Delta b - \Delta b'|}{|\Delta v|}. 
\]
A straightforward calculation shows that at deep renormalization levels we have that $A \leq 1$,
{\em c.f.} \cite[Lemma 2.34]{MP}.

\medskip

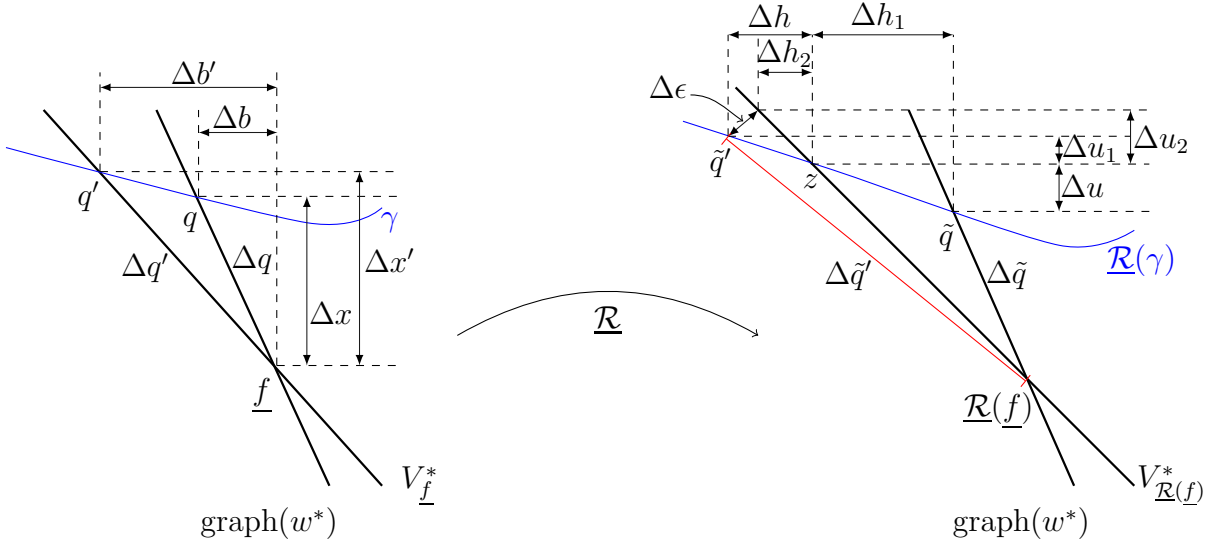
\begin{figure}
\label{fig proposition C1 manifold}
	\begin{center}

\begin{tikzpicture}[baseline]

\path[draw, line width=0.3mm,top color=white!40,bottom
color=white!40] (4.3,0) -- (2,5);
\path[draw, line width=0.3mm,top color=white!40,bottom
color=white!40] (0.5,5) -- (5,0);
\draw [blue] plot [smooth] coordinates {(5,3.7) (4,3.5) (0,4.5)};

\draw [dashed] (3.6,1.6) -- (5.2,1.6);
\draw [dashed] (2.56,3.85) -- (5.2,3.85);
\draw [dashed] (1.25,4.18) -- (5.2,4.18);

\draw [dashed] (3.6,1.6) -- (3.6,5.5);
\draw [dashed] (2.56,3.85) -- (2.56,5);
\draw [dashed] (1.25,4.18) -- (1.25,5.5);



\draw[draw=black!90,latex-latex] (4,3.75) +(0,0.1cm) -- (4,1.7) -- +(0,-0.1cm);

\draw[draw=black!90,latex-latex] (4.7,4.08) +(0,0.1cm) -- (4.7,1.7) -- +(0,-0.1cm);



\draw[draw=black!90,latex-latex] (2.66,4.7) +(-0.1cm,0) -- (3.51,4.7) -- +(0.1cm,0);

\draw[draw=black!90,latex-latex] (1.35,5.3) +(-0.1cm,0) -- (3.51,5.3) -- +(0.1cm,0);

\draw[black] (3.5,-0.5) node {$\mathrm{graph}(w^*)$};
\draw[black] (5.5,0) node {$V_{\underline{f}}^*$};
\draw[black] (3.4,1.2) node {$\underline{f}$};
\draw[black] (2.45,3.5) node {$q$};
\draw[black] (1.1,3.85) node {$q'$};
\draw[black] (1.85,2.9) node {$\Delta q'$};
\draw[black] (3.27,3) node {$\Delta q$};
\draw[black] (4.32,2.3) node {$\Delta x$};
\draw[black] (5.1,3) node {$\Delta x'$};
\draw[black] (3,4.9) node {$\Delta b$};
\draw[black] (2.5,5.5) node {$\Delta b'$};
\draw[blue] (5.1,3.5) node {$\gamma$};

\path[draw, line width=0.3mm,top color=white!40,bottom
color=white!40] (9.7,5.3) -- (15,0);
\path[draw, line width=0.3mm,top color=white!40,bottom
color=white!40] (12,5) -- (14.2,0);
\draw [blue] plot [smooth] coordinates {(15,3.4) (14,3.2) (11,4.2)(9,4.85)};

\draw [dashed] (12.565,3.65) -- (15.2,3.65);
\draw [dashed] (10.7,4.28) -- (15.2,4.28);
\draw [dashed] (10,5) -- (15.2,5);
\draw [dashed] (9.6,4.65) -- (14.95,4.65);

\draw [dashed] (9.6,4.65) -- (9.6,6);
\draw [dashed] (10,5) -- (10,6);
\draw [dashed] (12.595,3.65) -- (12.595,6);
\draw [dashed] (10.72,4.28) -- (10.72,6);




\draw[draw=black!90,latex-latex] (14,3.75) +(0,-0.1cm) -- (14,4.18) -- +(0,0.1cm);

\draw[draw=black!90,latex-latex] (14,4.38) +(0,-0.1cm) -- (14,4.55) -- +(0,0.1cm);

\draw[draw=black!90,latex-latex] (14.95,4.38) +(0,-0.1cm) -- (14.95,4.9) -- +(0,0.1cm);






\draw[draw=black!90,latex-latex] (9.7,6) +(-0.1cm,0) -- (10.62,6) -- +(0.1cm,0);

\draw[draw=black!90,latex-latex] (10.82,6) +(-0.1cm,0) -- (12.495,6) -- +(0.1cm,0);

\draw[draw=black!90,latex-latex] (10.1,5.5) +(-0.1cm,0) -- (10.62,5.5) -- +(0.1cm,0);

\draw[draw=black!90,latex-latex] (9.7,4.75) +(-0.1cm,-0.1cm) -- (9.9,4.9) -- +(0.1cm,0.1cm);


\draw[{|[right]}-{|[left]}, red] (9.57,4.62) to (13.56,1.39);

\draw[black] (13.5,-0.5) node {$\mathrm{graph}(w^*)$};
\draw[black] (15.5,0) node {$V_{\underline{\mathcal{R}}(\underline{f})}^*$};
\draw[black] (13.2,1) node {$\underline{\mathcal{R}}(\underline{f})$};
\draw[black] (12.5,3.3) node {$\tilde{q}$};
\draw[black] (9.5,4.3) node {$\tilde{q}'$};
\draw[black] (10.7,4.05) node {$z$};
\draw[black] (11.2,2.8) node {$\Delta \tilde{q}'$};
\draw[black] (13.3,2.8) node {$\Delta \tilde{q}$};
\draw[black] (14.31,3.99) node {$\Delta u$};
\draw[black] (14.41,4.48) node {$\Delta u_1$};
\draw[black] (15.36,4.6) node {$\Delta u_2$};
\draw[black] (10.15,6.25) node {$\Delta h$};
\draw[black] (11.6,6.25) node {$\Delta h_1$};
\draw[black] (10.35,5.75) node {$\Delta h_2$};
\draw[black] (8.8,5.3) node {$\Delta \epsilon $};
\draw[blue] (15.1,3) node {$\underline{\mathcal R}(\gamma)$};

\draw [<-,black] (9.8,4.85) to [out=145,in=0] (9.1,5.2);

\draw[black] (8,2.2) node {$\underline{\mathcal R}$};
\draw [->,black] (6,2) to [out=30,in=150] (10,2);

\end{tikzpicture}

\caption{Notation for the proof of Proposition~\ref{prop:deep manifold}.}
\end{center}
\end{figure}

\noindent\textit{Proof
of Proposition~\ref{prop:deep manifold}}.
We show that at a 
deep level of renormalization, each point $\underline{f} \in \mathrm{graph(w^*)}$ has a tangent plane $T_{\underline{f}}w^* = V_{\underline{f}}^*$.
To get this result it is enough to show that $A=0$.  We use the notation from the definition of $A$ and we introduce the following ones. 

\begin{displaymath}
\begin{array}{rcl}
     \underline{\mathcal{R}}(\underline{f})& = & (\tilde{b}, \tilde{x}), \\
     \underline{\mathcal{R}} (\gamma) \cap \mathrm{graph}(w^*) & = & \tilde{q} = \underline{\mathcal{R}}(q) = \underline{\mathcal{R}}(\underline{f}) + \Delta \tilde{q} = \underline{\mathcal{R}}(\underline{f}) + \big( \Delta \tilde{b}, \Delta \tilde{x} \big), \\
     \underline{\mathcal{R}}(\gamma) \cap V_{\underline{\mathcal{R}}(\underline{f})}^{*} & = & z = \underline{\mathcal{R}}(\underline{f})+ \Delta z, \\
    \underline{\mathcal{R}}(q') & = & \tilde{q}'= \underline{\mathcal{R}}(\underline{f})+ \Delta \tilde{q}' = \underline{\mathcal{R}}(\underline{f})+ \big( \Delta \tilde{b}', \Delta \tilde{x}'\big), \\
     z-\tilde{q} & = & \big( \Delta h_1, \Delta u \big), \\
     \tilde{q}' - z & = & \big( \Delta h, \Delta u_1 \big), \\
     \Delta \tilde{q}' & = & D \underline{\mathcal{R}}_{\underline{f}} \big( \Delta q'\big) + \Delta \epsilon,   \\
    D \underline{\mathcal{R}}_{\underline{f}} \big( \Delta q'\big) - \Delta z & = &  \big( \Delta h_2, \Delta u_2 \big).
\end{array}
\end{displaymath}
For almost horizontal curves $\gamma$ such that $\gamma \cap \mathrm{graph}(w^*)$ is close enough to $p$ we get

\begin{equation}
\label{epsilon much smaller than q'}
|\Delta \epsilon |  =  o \big( |\Delta q'| \big), 
\end{equation}
and
\begin{equation}
\label{epsilon greater than u_2-u_1}
|| \Delta u_2 - \Delta u_1||  \leq  |\Delta \epsilon|.   
\end{equation}
Since $\underline{\mathcal{R}}$ has strong expansion on $b$ direction, and using the differentiability of $\underline{\mathcal{R}}$, we get
\begin{equation}
\label{strong expansion on b direction}
|\Delta h_1|+|\Delta h| \geq  \displaystyle \frac{1}{|I'|} \cdot |\Delta b - \Delta b'|.    
\end{equation}
As $(\Delta h_2, \Delta u_2) \in V_{\underline{f}}^*$ and $V_{\underline{f}}^*$ is an admissible plane we get

\begin{equation}
\label{vector in V_R}
|\Delta h_2| \leq 2 \chi \tilde{b} ||\Delta u_2 ||.
\end{equation}
Since $q'-q = (\Delta b'- \Delta b, \Delta x'- \Delta x)$ is a tangent vector to the curve $\gamma$ it is inside the cone $C_{r, \delta}$, and then we get 

\begin{equation}
\label{tangent vector to gamma} 
||\Delta x'- \Delta x || < |\Delta b'- \Delta b|.
\end{equation}
As $(\Delta h, \Delta u_1)$ is a tangent vector to the curve $\underline{\mathcal{R}}_{\underline{f}}(\gamma)$, by the same reason as before, we get

\begin{equation}
\label{tangent vector to R gamma}
||\Delta u_1 || < |\Delta h|.
\end{equation}
By (\ref{vector in V_R}), (\ref{epsilon greater than u_2-u_1}) and (\ref{tangent vector to R gamma}) we have

\begin{equation*}
 \begin{array}{lcl}
   |\Delta h | & \leq & |\Delta \epsilon| + |\Delta h_2| \leq |\Delta \epsilon| + 2 \chi \tilde{b} || \Delta u_2|| \\
   & \leq  & |\Delta \epsilon| + 2 \chi \tilde{b} || \Delta u_2-\Delta u_1|| + 2 \chi \tilde{b} ||\Delta u_1||\\
   & \leq &  |\Delta \epsilon| + 2 \chi \tilde{b} |\Delta \epsilon|+  2 \chi \tilde{b} |\Delta h|. \\
 \end{array}
\end{equation*}
Hence, when we are in a deep level of renormalization we have

\begin{equation}
\label{h smaller than epsilon}
|\Delta h| \leq 2 |\Delta \epsilon|.
\end{equation}
Since 

\[
|\Delta q'| \leq |\Delta q| + ||q'-q|| = |\Delta q| + ||\Delta x' - \Delta x|| + |\Delta b'- \Delta b|,
\]
from (\ref{epsilon much smaller than q'}) and (\ref{tangent vector to gamma}) we obtain

\[
|\Delta \epsilon| = o \big( |\Delta q| + 2 |\Delta b'- \Delta b| \big) = o \big( |\Delta q| \big(1+2A \big) \big).  
\]
Hence

\begin{equation}
\label{epsilon much smaller than q}
|\Delta \epsilon| = o \big( |\Delta q| \big). 
\end{equation}
From this and using Lemma~\ref{Technical Lemma} we have

\begin{equation*}
\begin{array}{ccl}
   \displaystyle   \frac{|\Delta b - \Delta b'|}{|\Delta v|} & \leq &   \displaystyle  \frac{C_1 \cdot |I'| \cdot \big( |\Delta h_1|+|\Delta h|\big) }{|\Delta v|}  = C_1 \cdot |I'| \cdot \frac{|\Delta h_1| }{|\Delta v|}  + C_1 \cdot |I'| \cdot \frac{|\Delta h| }{|\Delta v|} \\
   & & \\
   & = & \displaystyle C_1 \cdot |I'| \cdot \frac{|\Delta \tilde{v}|}{|\Delta v|} \cdot \frac{|\Delta h_1| }{|\Delta \tilde{v}|}  + C_1 \cdot |I'| \cdot \frac{|\Delta h| }{|\Delta v|} \\
   & & \\ 
   & \leq & C_2 \cdot |I'| \cdot \displaystyle \frac{|\Delta h_1|}{|\Delta \tilde{v}|} + o(1), 
\end{array}
\end{equation*}
for a constant $C_2>0$. Hence we obtain

\[
\displaystyle \limsup_{\gamma \rightarrow p} \frac{|\Delta b - \Delta b'|}{|\Delta v|} \leq O(|I'|) A.
\]
Since $|I'|$ goes to zero when the level of renormalization goes to infinity we conclude that $A=0$, as desired.
\qed

Thus we have proved that there is an absorbing set, $B_0,$ for the renormalization operator within which the topological conjugacy class of $\underline f$ is a $\mathcal C^1$ manifold. It remains to prove that it is globally $\mathcal C^1$.

By \cite[Lemma 5.1]{AlvesColli2},
each infinitely renormalizable gap mapping $f_0=(f_{R},f_L,b_0)$ can be included in a family $f_t,$ for $t\in (-\varepsilon_0,\varepsilon_0)$
of gap mappings, which is transverse to the topological conjugacy class of $f_0.$ The construction of this family is given by varying the $b$ parameter in a small neighbourhood about $b_0$, and observing that the boundary points of the principal gaps at each renormalization level are strictly increasing functions in $b$. Thus we have that the transversality of this family is preserved under renormalization.
Let $\Delta f$ denote a vector tangent to the family 
$f_t$ at $f.$ We have the following:

\begin{lemma}
\label{lem:trans}
Let $n_0\in\mathbb N$ be so that
$\underline{\mathcal{R}}^{n_0}( \underline f)\in B_0.$  Then
$D\underline{\mathcal{R}}^{n_0}(\Delta f)\notin T_{\underline{\mathcal{R}}^{n_0} f} \mathrm{graph}(w^*),$
where $w^*=\mathcal T_{R^{n_0}(\underline f)}\cap B_0.$
\end{lemma}

Using this lemma, we can argue as in the proof of \cite[Theorem 9.1]{dFdMP} to conclude the proof Theorem~\ref{thm:A}:

\begin{theorem}
$\mathcal T_f\subset\mathcal D^4$ is a $\mathcal C^1$ manifold.
\end{theorem}
Note that the application of the Implicit Function Theorem in the proof is why we lose one degree of differentiability.

\section*{Acknowledgements}
The authors would like to Liviana Palmisano for some helpful comments about \cite{MP}. They also thank Sebastian van Strien for his encouragement and several
helpful conversations. 

\bibliographystyle{plain}
  \bibliography{biblio}

\end{document}